\documentclass[12pt]{amsart}
\usepackage{amsmath,amssymb,amscd,amsthm,amsfonts}
\usepackage{dsfont}

\usepackage[unicode]{hyperref}
\usepackage [numbers] {natbib}
\usepackage{enumerate}
\usepackage{mathtools}

\newtheorem{thm}{Theorem}[section]

\newtheorem{prop}[thm]{Proposition}
\newtheorem{cor}[thm]{Corollary}
\newtheorem{ass}[thm]{Assumption}
\theoremstyle{definition}
 
\theoremstyle{remark}

\newtheorem{rmk}[thm]{Remark}
\numberwithin{equation}{section}

\newcommand{\fa}{\mathfrak{a}}

\renewcommand{\AA}{\mathbb{A}}
\newcommand{\CC}{\mathbb{C}}

\newcommand{\RR}{\mathbb{R}} 
\newcommand{\ZZ}{\mathbb{Z}}

\newcommand{\cA}{\mathcal{A}}

\newcommand{\cK}{\mathcal{K}}

\newcommand{\cS}{\mathcal{S}}

\newcommand{\rO}{\mathrm{O}}

\newcommand{\disc}{\mathrm{disc}}

\newcommand{\half}{\frac {1} {2}}

\newcommand{\til}{\widetilde}
\DeclareMathOperator{\BC}{\mathrm{BC}}

\DeclareMathOperator{\sgn}{\mathrm{sgn}}

\DeclareMathOperator{\Rat}{\mathrm{Rat}}

\DeclareMathOperator{\GL}{\mathrm{GL}}
\DeclareMathOperator{\RGL}{\mathrm{RGL}}
\DeclareMathOperator{\SL}{\mathrm{SL}}

\DeclareMathOperator{\rU}{\mathrm{U}}
\DeclareMathOperator{\Sym}{\mathrm{Sym}}

\DeclareMathOperator{\Sp}{\mathrm{Sp}}
\DeclareMathOperator{\Mp}{\mathrm{Mp}}
\DeclareMathOperator{\SO}{\mathrm{SO}}

\DeclareMathOperator{\tr}{\mathrm{tr}}

\DeclareMathOperator{\Ind}{\mathrm{Ind}}

\DeclareMathOperator{\Gal}{\mathrm{Gal}}

\DeclareMathOperator{\R}{\mathrm{R}} %restriction of scalars
\DeclareMathOperator{\RS}{\mathrm{R}} %restriction of scalars

\newcommand{\isom}{\cong}

\DeclareMathOperator{\lmod}{\backslash}

\DeclareMathOperator{\LO}{\mathrm{LO}}
\DeclareMathOperator{\FO}{\mathrm{FO}} 
\let\Re\undefined
\DeclareMathOperator{\Re}{\mathrm{Re}}
\newcommand{\Asai}{\mathrm{Asai}}

\newcommand{\form}[2]{\langle{#1},{#2}\rangle}

\newcommand{\IARep}{irreducible automorphic representation}
\newcommand{\ICARep}{irreducible cuspidal automorphic representation}

\newcommand{\IUCARep}{irreducible unitary cuspidal automorphic representation}
\newcommand{\IUCAReps}{irreducible unitary cuspidal automorphic representations}

\let\RGL\undefined
\DeclareMathOperator{\RGL}{\mathrm{R}_{E/F}\GL}
\newcommand{\bG}{\mathbf{G}}
\usepackage[left=1.5in,right=1.5in, top=1.5in, bottom=1.5in]{geometry}
\usepackage[mathlines,displaymath]{lineno}
\begin{document}
\author{Chenyan Wu}
\email{chenyan.wu@unimelb.edu.au}
\address{School of Mathematics and Statistics
The University of Melbourne, Victoria 3010, Australia}

\thanks{Corresponding author.}
\date{\today}

\title[Theta correspondence and global Arthur parameters]{Theta correspondence and simple factors in global Arthur parameters}

\begin{abstract}
  By using results on poles of $L$-functions and theta correspondence, we give a bound on $b$ for $(\chi,b)$-factors of the global Arthur parameter of a cuspidal automorphic representation $\pi$ of a classical group or a metaplectic group where $\chi$ is a conjugate self-dual automorphic character and $b$ is an integer which is the dimension of an irreducible representation of $\SL_{2}(\CC)$.
We derive a more precise relation when $\pi$ lies in a generic global $A$-packet.
\end{abstract}

\maketitle{}
\section*{Introduction}
\label{sec:introduction}

Let $F$ be a number field and let $\AA$ be its ring of adeles.
Let $\pi$ be an \ICARep{} of a classical group $G$ defined over $F$.
We also treat the case of metaplectic groups in this work.
However to avoid excessive notation, we focus on the case of the symplectic groups $G=\Sp(X)$ in this introduction where $X$ is a non-degenerate symplectic space over $F$.
By Arthur's theory of endoscopy \cite{Arthur-book-MR3135650}, $\pi$ belongs to a global $A$-packet associated to an elliptic global $A$-parameter, which is of the form
\begin{equation*}
  \boxplus_{i=1}^{r} (\tau_i,b_i)
\end{equation*}
where $\tau_i$ is an irreducible  self-dual cuspidal automorphic representation of $\GL_{n_i} (\AA)$ and $b_i$ is a positive integer which represents the unique $b_i$-dimensional irreducible representation of Arthur's $\SL_2 (\CC)$.
See Section~\ref{sec:glob-arth-param}, for more details.

In  \cite{jiang14:_autom_integ_trans_class_group_i}, Jiang proposed the $(\tau,b)$-theory.
See, in particular, Principle~1.2 there.
It is a conjecture that uses period integrals to link together automorphic representations in two global $A$-packets whose global $A$-parameters are ``different'' by a $(\tau,b)$-factor.
% The statement there also deals with more general factors, but we will consider $(\tau,b)$-factors only.
We explain in more details.
Let $\Pi_{\phi}$ denote the global $A$-packet with elliptic global $A$-parameter $\phi$.
 Let $\pi$  be an \IARep{} of $G(\AA)$ and  let $\sigma$  be an \IARep{} of $H(\AA)$, where $H$ is a factor of an endoscopic group of $G$.
Assume that $\pi$ (resp. $\sigma$) occurs in the discrete spectrum.
Then it is expected that there exists some kernel function $\cK$ depending on $G$, $H$ and $(\tau,b)$ only such that
if $\pi$ and $\sigma$ satisfy  a Gan--Gross--Prasad type of criterion, namely, that the period integral
\begin{equation}\label{eq:jiang-integral}
  \int_{H(F)\lmod H(\AA)} \int_{G(F)\lmod G(\AA)} \cK(h,g) f_{\sigma}(h) \overline{f_{\pi}(g)} dg dh
\end{equation}
is non-vanishing for some choice of $f_{\sigma}\in\sigma$ and $f_{\pi}\in\pi$, then
$\pi$ is in the global $A$-packet $\Pi_{\phi}$ if and only if $\sigma$ is in the global $A$-packet  $\Pi_{\phi_{2}}$ with $\phi=(\tau,b)\boxplus \phi_{2}$.
Then \cite[Section~5]{jiang14:_autom_integ_trans_class_group_i} proceeds to construct certain kernel functions and then using them, defines endoscopy transfer (by integrating over $H(F)\lmod H(\AA)$ only in~\eqref{eq:jiang-integral}) and endoscopy descent (by integrating over $G(F)\lmod G(\AA)$ only in~\eqref{eq:jiang-integral}).
It is not yet known if these are the kernel functions making the statements of Principle~1.2 in \cite{jiang14:_autom_integ_trans_class_group_i} hold.
As the kernel functions come from Bessel coefficients or Fourier--Jacobi coefficients as in \cite[Section~23]{GGP-MR3202556}, we see the non-vanishing of  this period integral is analogous to condition~(i) in the global Gan--Gross--Prasad Conjecture~\cite[Conjecture~24.1]{GGP-MR3202556}.

Jiang suggested in  \cite[Section~7]{jiang14:_autom_integ_trans_class_group_i} that if $\tau$ is an automorphic character $\chi$, then the kernel function can be taken to be the theta kernel and endoscopy transfer and endoscopy descent are theta lifts.
In this case, the span of 
\begin{align*}
  \int_{G(F)\lmod G(\AA)} \cK(h,g) \overline{f_{\pi}(g)} dg
\end{align*}
as $f_{\pi}$ runs over $\pi$
is the theta lift of $\pi$.
This is an automorphic representation of $H(\AA)$.
Lifting in the other direction is analogous.
%This is part of our future work.
% Much work is still needed to attack in \cite{jiang14:_autom_integ_trans_class_group_i} even in the $(\chi,b)$-case.
% Our future work is aimed at  determining the global $A$-parameter of the theta lift of $\pi$ from $G$ to $H$.
Assume that the theta lift of $\pi$ is non-zero.
Write $\phi_{\pi}$ for the global $A$-parameter of $\pi$.
Then \cite[Principle~1.2]{jiang14:_autom_integ_trans_class_group_i} says that $\phi_{\pi}$  has a $(\chi,b)$-factor and that the global $A$-parameter of the theta lift of $\pi$ from $G$ to $H$ should be $\phi_{\pi}$ with the $(\chi,b)$-factor removed.
Here $b$ should be of  appropriate size relative to $G$ and $H$.
 Our work is one step in this direction.
%Much work is still needed to attack this.
% For example, it is not clear that there exists such a $(\chi,b)$-factor in the global $A$-parameter of $\pi$.

One goal of this article is to expand on the $(\chi,b)$-theory and to present the results of \cite{Moeglin-eis-pole-theta-MR1473165}, \cite{GJS-MR2540878}, \cite{JW-sympletic-MR3805648}, \cite{wu22:_period-metaplectic-JNT-in-press}, \cite{JW-unitary-MR3435720}, \cite{wu22:_poles_eisen-unitary-PJM}   for various cases in a uniform way.
As different reductive dual pairs that occur in theta correspondence have their own peculiarities, the notation and techniques  of these papers are adapted to the treatment of their own specific cases. We attempt to emphasise on the common traits of the results which are buried in lengthy and technical proofs in these papers.

After  collecting the results on poles of $L$-functions, poles of Eisenstein series and theta correspondence, we derive a bound for $b$ when $b$ is maximal among all factors of the global $A$-parameter of $\pi$.
In addition,  we derive an implication on global $A$-packets. Of course, the heavy lifting was done by the papers mentioned above.

\begin{thm}[Corollary~\ref{cor:A-packet-no-cusp}]
  The global $A$-packet attached to the elliptic global $A$-parameter $\phi$ cannot have a cuspidal member if $\phi$ has a $(\chi,b)$-factor with
  \begin{equation*}
    b > \half \dim_F X + 1,  \quad \text{if $G=\Sp(X)$}.
  \end{equation*}
\end{thm}

Another way of phrasing this is that we have a bound on the size of $b$ that can occur in a factor of type $(\chi,*)$ in the global $A$-parameter of a cuspidal automorphic representation.
Thus our results have application in getting a Ramanujan bound, which measures the departure of the local components $\pi_{v}$ from being tempered for all places $v$ of $F$, for classical groups and metaplectic groups.
This should follow by generalising the arguments in \cite[Section~5]{Jiang-Liu-MR3969876} which treats the symplectic case.
There they first established a bound for $b$ under some conditions on wave front sets.
This enables them to control the contribution of $\GL_{1}$-factors in the global $A$-parameter to the Ramanujan bound.
Our result can supply this ingredient for  classical groups and also metaplectic groups unconditionally.
Then \cite[Section~5]{Jiang-Liu-MR3969876} found a Ramanujan bound for $\pi$ by using the crucial results on the Ramanujan bound for $\GL_{2}$ in \cite{Kim-MR1937203} and \cite{Blomer-Brumley-MR2811610}.
% The implication on $A$-global packets is the base case for estimating the departure from Ramanujan Conjection as in \cite{Jiang-Liu-MR3969876}.

We describe the idea of the proof of our result.
First we relate the existence of a $(\tau,b)$-factor in the elliptic global $A$-parameter of $\pi$ to the existence of poles of partial $L$-functions $L^{S}(s,\pi\times\tau^{\vee})$. See Proposition~\ref{prop:tau-b-L-pole}.
If the global $A$-parameter of $\pi$ has a factor $(\tau,b)$ where  $b$  is maximal among all factors, we can show that the partial $L$-function $L^{S}(s,\pi\times \tau^{\vee})$ has a pole at $s=\half(b+1)$. Thus studying the location of poles of $L^{S}(s,\pi\times \tau^{\vee})$ for $\tau$ running through all  self-dual cuspidal representations of $\GL_{n}(\AA_{F})$ can shed light on the size of $b_{i}$'s that occur in the global $A$-parameter of $\pi$.
Then we specialise to the case where $\tau$ is a character $\chi$ and consider  $L^{S}(s,\pi\times \chi^{\vee})$ in what follows. % When $\tau$ is an automorphic character $\chi$, the $L$-function has been well-studied and is intricately entwined with the study of theta correspondence. It occurs in the Rallis inner product formula which says that the inner product of two theta lifts is equal to the residue or value of $L(s,\pi\times\chi^{\vee})$ at an appropriate point up to some ramified factors and
% some abelian $L$-functions. We refer the reader to \cite{Yamana-Lfun-MR3211043} which is a culmination of many previous results. See also the references in \cite{Yamana-Lfun-MR3211043}.

Next we relate the poles of $L^{S}(s,\pi\times\chi^{\vee})$ to the poles of Eisenstein series attached to the cuspidal datum $\chi\boxtimes\pi$. See Section~\ref{sec:eisenst-seri}.
In fact, in some cases, we use the non-vanishing of $L^{S}(s,\pi\times\chi^{\vee})$ at $s=\half$ instead. See Proposition~\ref{prop:L-pole-Eis-pole}. Then we recall in Theorem~\ref{thm:max-pole-eis} that the maximal positive pole of the Eisenstein series has a bound which is supplied by the study of global theta lifts. This is enough for showing Corollary~\ref{cor:A-packet-no-cusp}, though we have a more precise result that the maximal positive pole corresponds to the invariant called the lowest occurrence index of $\pi$ with respect to $\chi$ in Theorem~\ref{thm:LO-equiv-eis-max-pole}.
The lowest occurrence index is the minimum of the first occurrence indices over some Witt towers.
For the precise definition see~\eqref{eq:defn-LO}.
We also have a less precise result (Theorem~\ref{thm:FO-eis-pole}) relating  the first occurrence index of $\pi$ with respect to certain quadratic spaces to possibly non-maximal and possibly negative poles of the Eisenstein series.

More precise results can be derived if we assume that $\pi$ has a generic global $A$-parameter. % Then all factors of the global $A$-parameter of $\pi$ is of the form $(*,1)$.
This is because we have a more precise result relating poles or non-vanishing of values of the complete $L$-functions to poles of the Eisenstein series supplied by \cite{JLZ-MR3079762}. Thus we get

\begin{thm}[Theorem~\ref{thm:chi-b-generic}]
  Let $\pi$ be a cuspidal member in a generic global $A$-packet    of $G(\AA)=\Sp(X)(\AA)$. Let $\chi$ be a  self-dual automorphic character of $\GL_{1}(\AA)$. Then 
  the following are equivalent.
  \begin{enumerate}
  \item The global $A$-parameter $\phi_{\pi}$ of $\pi$ has a $(\chi,1)$-factor.
  \item The complete $L$-function $L(s,\pi\times\chi^{\vee})$ has a pole at $s=1$.
  \item The Eisenstein series $E(g,f_{s})$ has a pole at $s=1$ for some choice of section $f_{s}\in \cA^{Q_{1}}(s,\chi\boxtimes\pi)$.
  \item The lowest occurrence index $\LO_{X}^{\chi}(\pi)$ is $\dim X$.
  \end{enumerate}
\end{thm}

Here $Q_{1}$ is a parabolic subgroup of $\Sp(X_{1})$ with Levi subgroup  isomorphic to $\GL_{1}\times \Sp(X)$, where $X_{1}$ is the symplectic space formed from $X$ by adjoining a hyperbolic plane.
Roughly speaking, $\cA^{Q_{1}}(s,\chi\boxtimes\pi)$ is a space of automorphic forms on $\Sp(X_{1})$ induced from  $\chi|\ |^{s}\boxtimes\pi$ viewed as a representation of the parabolic subgroup $Q_{1}$.
We refer the reader to Section~\ref{sec:eisenst-seri} for the precise definition of $\cA^{Q_{1}}(s,\chi\boxtimes\pi)$.
We note that the lowest occurrence index $\LO_{X}^{\chi}(\pi)$ is an invariant in the theory of theta correspondence related to the invariant called the first occurrence index. See Section~\ref{sec:theta-correspondence} for their definitions.
We also include a result (Theorem~\ref{thm:generic-param-L-pole-Eis-pole}) that concerns the non-vanishing of $L(s,\pi\times\chi^{\vee})$ at $s=\half$ and the lowest occurrence index. We plan to improve this result in the future by studying a relation between non-vanishing of Bessel or Fourier--Jacobi periods and the lowest occurrence index. % Finally we  also plan to tackle  the case of $(\tau,1)$ for $\tau$ not just a character. [later]

We note that the $L$-function $L(s,\pi\times\chi^{\vee})$ has been well-studied and is intricately entwined with the study of theta correspondence, most prominently in the Rallis inner product formula which says that the inner product of two theta lifts is equal to the residue or value of $L(s,\pi\times\chi^{\vee})$ at an appropriate point up to some ramified factors and some abelian $L$-functions. We refer the reader to \cite{Yamana-Lfun-MR3211043} which is a culmination of many previous results. See also the references in \cite{Yamana-Lfun-MR3211043}.
In our approach, the Eisenstein series $E(g,f_{s})$, which is not of Siegel type, is the key link between $L(s,\pi\times\chi^{\vee})$ and the theta lifts.

Now we describe the structure of this article.
In Section~\ref{sec:notation}, we set up some basic notation.
In Section~\ref{sec:glob-arth-param}, we define elliptic global $A$-parameters for classical groups and metaplectic groups and also the global $A$-packet associated to an elliptic global $A$-parameter.
We show how poles of partial $L$-functions detect $(\tau,b)$-factors in an elliptic global $A$-parameter.
In Section~\ref{sec:eisenst-seri}, we define Eisenstein series attached to the cuspidal datum $\chi\boxtimes\pi$ and recall some results on the possible locations of their maximal positive poles.
In Section~\ref{sec:theta-correspondence}, we introduce two invariants of theta correspondence.
They are the first occurrence index $\FO_{X}^{Y,\chi}(\pi)$ and the lowest occurrence index $\LO_{X}^{\chi}(\pi)$ of $\pi$ with respect to some data.
We relate them to poles of Eisenstein series.
Results in Sections~\ref{sec:eisenst-seri} and~\ref{sec:theta-correspondence} are not new.
Our aim is to present the results in a uniform way for easier access.
In Section~\ref{sec:appl-glob-arth}, we show a bound for $b$ in $(\chi,b)$-factors of the global $A$-parameter of $\pi$.
Finally in Section~\ref{sec:gener-repr}, we consider the case when $\pi$ has a generic global $A$-parameter.
We show that when $L(s,\pi\times\chi^{\vee})$ has a pole at $s=1$ (resp.
$L(s,\pi\times\chi^{\vee})$ is non-vanishing at $s=\half$),  the lowest occurrence index is determined.

\section*{Acknowledgement}
\label{sec:acknowledgement}
The author would like to thank Professor Dihua Jiang for suggesting the topic to her.
The author would also like to thank the referee for reading the manuscript carefully  and for informing her of some recent developments.

\section{Notation}
\label{sec:notation}

% I think of $E$ as the centre of a division algebra $D$ with an anti-involution $\varrho$. $\varrho$ restricts to an automorphism of $E$ and $F$ is the fixed point set of $\varrho$ in $E$. Will work with $D$ in the future.
Let $F$ be a number field and let $E$ be either $F$ or a quadratic field extension of $F$. Let $\varrho\in\Gal(E/F)$ be the trivial Galois element when $E=F$ and the non-trivial Galois element when $E\neq F$. When $E\neq F$, write $\varepsilon_{E/F}$ for the quadratic character associated to $E/F$ via Class field theory. Let $G$ be an algebraic group over $E$. We write $\R_{E/F}G$ for the restriction of scalars of Weil. This is an algebraic group over $F$.

Let $\epsilon$ be either $1$ or $-1$.  By an $\epsilon$-skew Hermitian space, we mean an $E$-vector space $X$ together with an $F$-bilinear pairing
\begin{align*}
  \form{\ }{\ }_X : X \times X \rightarrow E
\end{align*}
such that
\begin{equation*}
  \form{y}{x}_X = -\epsilon\form{x}{y}_{X}^\varrho, \quad \form{ax}{by} = a\form{x}{y}_{X}b^\varrho 
\end{equation*}
for all $a,b\in E$ and $x,y\in X$. We consider the linear transformations of $X$ to act from the right.
We  follow \cite{Yamana-Lfun-MR3211043}'s notation closely and we intend to generalise the results here to the quaternionic unitary group case in our future work.

Let $X$ be an $\epsilon$-skew Hermitian space of finite dimension. Then the isometry group of $X$ is one of the following:
\begin{enumerate}
\item the symplectic group $\Sp(X)$ when $E=F$ and $\epsilon=1$;
\item the orthogonal group $\rO(X)$ when $E=F$ and $\epsilon=-1$;
\item the unitary group $U(X)$ when $E\neq F$ and $\epsilon=\pm 1$.
\end{enumerate}

We will also consider the metaplectic group.
Let $v$ be a place of $F$ and let $F_{v}$ denote the completion of $F$ at $v$. Let $\AA_{F}$ (resp. $\AA_{E}$) denote the ring of adeles of $F$ (resp. $E$). Set $\AA:=\AA_{F}$.
Write $\Mp(X)(F_v)$ (resp. $\Mp(X)(\AA_F)$) for the metaplectic double cover of $\Sp(X)(F_v)$ (resp. $\Sp(X)(\AA_F)$) defined by Weil \cite{Weil-MR0165033}. We note that the functor $\Mp(X)$  is not representable by an algebraic group. We will also need the $\CC^{1}$-extension $\Mp(X)(F_{v})\times_{\mu_{2}} \CC^{1}$ of $\Sp(X)(F_v)$  and we denote it by $\Mp^{\CC^{1}}(F_{v})$. Similarly we define $\Mp^{\CC^{1}}(\AA_{F})$.

Let $\psi$ be a non-trivial automorphic additive character of $\AA_F$ which will figure in the Weil representations as well as the global $A$-parameters for $\Mp(X)$. 

For an automorphic representation or admissible representation $\pi$, we write $\pi^{\vee}$ for its contragredient.

\section{Global Arthur Parameters}
\label{sec:glob-arth-param}

First we recall the definition of elliptic global Arthur parameters ($A$-parameters) for classical groups as well as metaplectic groups.
See \cite{Arthur-book-MR3135650} for the symplectic and the special orthogonal case and we adopt the formulation in \cite{Atobe-Gan-MR3708200} for  the case of the (disconnected) orthogonal groups.
For the unitary case, see \cite{Mok-MR3338302,kaletha:_endos_class_repres}.
For the metaplectic case, see  \cite{Gan-Ichino-Mp-MR3866889}.
Then we  focus on  simple factors of global Arthur parameters and relate their presence to poles of partial $L$-functions.
This is a crude first step for detecting $(\tau,b)$-factors in an elliptic global $A$-parameter according to the `$(\tau,b)$-theory'  proposed in \cite{jiang14:_autom_integ_trans_class_group_i}.

Let $\bG$ be  $\rU(X)$, $\rO(X)$, $\Sp(X)$ or $\Mp(X)$.
Let $d$ denote the dimension of $X$.
Set $\bG^{\circ}=\SO(X)$  when $\bG=\rO(X)$.
Set $\bG^{\circ}=\bG$ otherwise.
Write $\check{\bG}$ for the (complex) dual group of $\bG^{\circ}$.
Then
\begin{equation*}
  \check{\bG}=
  \begin{cases}
    \GL_{d}(\CC), \quad&\text{if $\bG=\rU(X)$};\\
    \Sp_{d-1}(\CC), \quad&\text{if $\bG=\rO(X)$ and $d$ is odd};\\
    \SO_{d}(\CC), \quad&\text{if $\bG=\rO(X)$ and $d$ is even};\\
    \SO_{d+1}(\CC), \quad&\text{if $\bG=\Sp(X)$};\\
    \Sp_{d}(\CC), \quad&\text{if $\bG=\Mp(X)$}.
  \end{cases}
\end{equation*}
An elliptic global $A$-parameter for $\bG$ is a finite formal sum of the form
\begin{equation*}
  \phi = \boxplus_{i=1}^{r} (\tau_i,b_i), \quad\text{for some  positive integer  $r$}
\end{equation*}
where
\begin{enumerate}
\item  $\tau_i$ is an irreducible conjugate self-dual cuspidal automorphic representation of $\GL_{n_i} (\AA_E)$;
\item $b_i$ is a positive integer which represents the unique $b_i$-dimensional irreducible representation of Arthur's $\SL_2 (\CC)$
\end{enumerate}
such that 
\begin{itemize}
\item $\sum_{i} n_ib_i = d_{\check{\bG}}$;
\item  $\tau_i$ is conjugate self-dual of parity $(-1)^{N_{\check{\bG}}+b_i}$ (see Remark~\ref{rmk:conj-self-dual-and-L-fun-characterisation});
\item the factors $(\tau_i,b_i)$ are pairwise distinct.
\end{itemize}
Here $d_{\check{\bG}}$ is the degree of the standard representation of $\check{\bG}$ which, explicitly, is
\begin{equation*}
  d_{\check{\bG}}=
  \begin{cases}
    \dim X, & \text{ if $\bG=\rU(X)$};\\
    \dim X -1, & \text{ if $\bG=\rO(X)$ with $\dim X$ odd};\\
    \dim X , & \text{ if $\bG=\rO(X)$ with $\dim X$ even};\\
    \dim X +1, & \text{ if $\bG=\Sp(X)$};\\
    \dim X, & \text{ if $\bG=\Mp(X)$}
  \end{cases}
\end{equation*}
and 
\begin{equation*}
  N_{\check{\bG}} =
  \begin{cases}
    \dim X \mod 2, & \text{ if $\bG=\rU(X)$};\\
    0, & \text{ if $\bG=\rO(X)$ with $\dim X$ odd};\\
    1 , & \text{ if $\bG=\rO(X)$ with $\dim X$ even};\\
    1, & \text{ if $\bG=\Sp(X)$};\\
    0, & \text{ if $\bG=\Mp(X)$}.
  \end{cases}
\end{equation*}
\begin{rmk}
  We adopt the notation in \cite{jiang14:_autom_integ_trans_class_group_i} and hence we write  $(\tau_i,b_i)$ rather than $\tau_i\boxtimes \nu_{b_i}$ as is more customary in the literature, so that the quantity $b_i$, that we  study, is more visible.
\end{rmk}
% (Maybe related to $d_{D,\epsilon}$ in Sun-Zhu)
% We note that in the $\rO(X)$ case, the global $A$-parameters are just those for $\SO(X)$ as in the formulation of \cite{Atobe-Gan-MR3708200}.
% Need to identify those related by an out automorphism in the even orthogonal case.
\begin{rmk}
  In the unitary case, we  basically spell out what $\Psi_2(\rU(N),\xi_{\mathbf{1}})$ in \cite[Definition~2.4.7]{Mok-MR3338302} is.
  We have discarded the second factor $\tilde{\psi}$ as it is determined by $\psi^N$ and $\xi_{\mathbf{1}}$ in Mok's notation.
\end{rmk}

\begin{rmk}\label{rmk:conj-self-dual-and-L-fun-characterisation}
  \begin{enumerate}
  \item For $\bG=\rU(X)$, we say that $\tau$ is conjugate self-dual of parity $\eta$ if the Asai $L$-function $L(s,\tau,\Asai^\eta)$ has a pole at $s=1$. If $\eta=+1$, we also say that $\tau$ is conjugate orthogonal and if $\eta=-1$, we also say that $\tau$ is conjugate symplectic. The Asai representations come from the decomposition of the twisted tensor product representation of the $L$-group.  See \cite[(2.2.9) and (2.5.9)]{Mok-MR3338302} and \cite{Goldberg-MR1266747}.
  \item For other cases, we mean self-dual when we write conjugate self-dual. We say that $\tau$ is self-dual of parity $+1$ or orthogonal, if $L(s,\tau,\Sym^{2})$ has a pole at $s=1$; we say that $\tau$ is self-dual of parity $-1$ or symplectic, if $L(s,\tau,\wedge^{2})$ has a pole at $s=1$.
  \item The parity is uniquely determined for each irreducible conjugate self-dual cuspidal representation $\tau$.
  \end{enumerate}
\end{rmk}

Let $\Psi_{2}(\bG)$ denote the set of elliptic global $A$-parameters of $\bG$.
Let $\phi\in\Psi_{2}(\bG)$.
 Via the local Langlands conjecture (which is proved for the general linear groups), at every place $v$ of $F$, we  localise $\phi$ to get an elliptic local $A$-parameter,
\begin{equation*}
  \phi_{v}:L_{F_{v}} \times \SL_{2}(\CC) \rightarrow \check{\bG}\rtimes W_{F_{v}},
\end{equation*}
where  $W_{F_{v}}$ is the Weil group of $F_{v}$ and $L_{F_{v}}$ is $W_{F_{v}}$ if $v$ is archimedean and the Weil--Deligne group $W_{F_{v}}\times \SL_{2}(\CC)$ if $v$ is non-archimedean. To $\phi_{v}$ we associate the local $L$-parameter $\varphi_{\phi_{v}}:L_{F_{v}}\rightarrow \check{\bG}\rtimes W_{F_{v}}$ given by
\begin{equation*}
  \varphi_{\phi_{v}}(w)= \phi_{v}(w,
  \begin{pmatrix}
    |w|^{\half} & \\
                & |w|^{-\half}
  \end{pmatrix}
  ) .
\end{equation*}

Let $L^{2}_{\disc}(\bG)$ denote the discrete part of $L^{2}(\bG(F)\lmod \bG(\AA_{F}))$ when $\bG\neq \Mp(X)$ and the genuine discrete part of $L^{2}(\Sp(F)\lmod \Mp(\AA_{F}))$ for $\bG=\Mp(X)$.
Define the full near equivalence class $L^{2}_{\phi,\psi}(\bG)$ attached to  the elliptic global $A$-parameter $\phi$ to be the Hilbert direct sum of all irreducible automorphic representations $\sigma$ occurring in $L^2_\disc(\bG)$ such that for almost all $v$, the local $L$-parameter of $\sigma_v$ is $\varphi_{\phi_v}$.
We remark that in the $\Mp(X)$-case, the parametrisation of $\sigma_v$ is relative  to $\psi_{v}$ since the local $L$-parameter of $\sigma_v$ is attached via the Shimura--Waldspurger correspondence which depends on $\psi_{v}$.
This is the only case in this article where $L^{2}_{\phi,\psi}(\bG)$  depends on $\psi$.

Let $\cA_2(\bG)$ denote the dense subspace  consisting of automorphic forms in $L^2_\disc (\bG)$.
Similarly define $\cA_{2,\phi,\psi}(\bG)$ to be the dense subspace  of $L^2_{\phi,\psi} (\bG)$ consisting of automorphic forms.
Then we have a crude form of Arthur's multiplicity formula which decomposes the $L^2$-discrete spectrum into near equivalence classes indexed by $\Psi_{2}(\bG)$. 

\begin{thm}\label{conj:arthur-mult-formula-crude}
  We have the orthogonal decompositions
  \begin{align*}
    L^2_\disc (\bG)= \hat{\oplus}_{\phi\in \Psi_2(\bG)}L^2_{\phi,\psi}(\bG) \quad \text{and} \quad
    \cA_2(\bG)= \oplus_{\phi\in \Psi_2(\bG)} \cA_{2,\phi,\psi}(\bG).
  \end{align*}
\end{thm}

\begin{rmk}
  This crude form of Arthur's multiplicity formula  has been proved for $\Sp(X)$ and quasi-split $\rO(X)$ by Arthur \cite{Arthur-book-MR3135650}, for $\rU(X)$ by \cite{Mok-MR3338302} and \cite{kaletha:_endos_class_repres} and for $\Mp(X)$ by \cite{Gan-Ichino-Mp-MR3866889}.
  This is also proved for non-quasi-split even orthogonal (and also unitary groups) in \cite{Chen-Zou-arxiv-2103} and for non-quasi-split  odd orthogonal groups in \cite{Ishimoto-arxiv-2301}.
  Thus for all cases needed in this paper, Theorem~\ref{conj:arthur-mult-formula-crude} is known.
% The above mentioned works also treats refined Arthur's multiplicity formula. not complete.We only need this crude form in this paper refined Arthur's multiplicity formula There is 
%   It is expected to hold also for pure inner forms of $\rO(X)$ and $\Sp(X)$ in full generality.
  % We do not claim that  the near equivalent class that corresponds to $\phi\in \Psi_2(\bG)$ is non-zero.
\end{rmk}

We have some further remarks on the orthogonal and unitary cases.
\begin{rmk}
  Arthur's statements use $\SO(X)$ rather than $\rO(X)$ and he needs to account for the outer automorphism of $\SO(X)$ when  $\dim X$ is  even.
  See the paragraph below \cite[Thm.~1.5.2]{Arthur-book-MR3135650}.
  The formulation for quasi-split even $\rO(X)$ is due to Atobe and Gan \cite[Thm.~7.1(1)]{Atobe-Gan-MR3708200}.
  For odd $\rO(X)$, which is isomorphic to $\SO(X)\times\mu_2$, the reformulation  of Arthur's result is easy.
  Let $T$ be a finite set of places of $F$. Assume that it has even cardinality.
  Let $\sgn_T$ be the automorphic character of $\mu_2(\AA_F)$ which is equal to the sign character at places in $T$ and the trivial character at places outside $T$.
  These give all the automorphic characters of $\mu_2(\AA_F)$.
  Then every irreducible automorphic representation $\pi$ of $\rO(X)(\AA_F)$ is of the form $\pi_0\boxtimes \sgn_T$ for some irreducible automorphic representation $\pi_0$ of $\SO(X)(\AA_F)$ and some finite set $T$ of places of even cardinality.
  A near equivalence class of $\rO(X)(\AA_F)$ then consists of all irreducible automorphic representations $\pi_0\boxtimes \sgn_T$ for $\pi_0$ running over a near equivalence class of $\SO(X)(\AA_F)$  and $\sgn_T$ running over all automorphic characters of $\mu_2(\AA_F)$.
\end{rmk}

\begin{rmk}
  For the $\rU(X)$ case, the global $A$-parameter depends on the choice of a sign and a conjugate self-dual character which determine an embedding of the $L$-group of $\rU(X)$ to the $L$-group of $\RGL_d$ where we recall that $d:=\dim X$.
  We refer the reader to \cite[Sec.~2.1]{Mok-MR3338302}, in particular (2.1.9) there, for details.
  In this work, we choose the $+1$ sign and the trivial character, which, in Mok's notation, means $\kappa =1$ and $\chi_\kappa=\mathbf{1}$.
  Then this corresponds to the standard base change of $\rU(X)$ to $\RGL_d$.
  We note that the  $L$-functions we use below are such that
  \begin{equation*}
    L_v(s,\pi_v\times\tau_v) = L_v(s,\BC(\pi_v)\otimes \tau_v),
  \end{equation*}
  for all places $v$, automorphic representations $\pi$ of $\bG(\AA_{F})$ and $\tau$ of $\RGL_a(\AA_{F})$ where $\BC$ denotes the standard base change.
\end{rmk}

By Theorem~\ref{conj:arthur-mult-formula-crude}, we  get
\begin{prop}\label{prop:tau-b-L-pole}
  Let $\pi$ be an irreducible automorphic representation of $\bG(\AA_{F})$ that occurs in $\cA_{2,\phi,\psi}(\bG)$. Then
  \begin{enumerate}
  \item if $\phi$ has a $(\tau,b)$-factor with $b$ maximal among all factors, then the partial $L$-function $L^S(s,\pi\times \tau^\vee)$ has a pole at $s=\frac{b+1}{2}$ and this is its maximal pole;
  \item if the partial $L$-function $L^S(s,\pi\times \tau^\vee)$ has a pole at $s=\frac{b'+1}{2}$, then $\phi$ has a $(\tau,b)$-factor with $b\ge b'$.
  \end{enumerate}
\end{prop}
\begin{rmk}
  In the $\Mp(X)$ case, the $L$-function depends on $\psi$, but we suppress it from notation here.
\end{rmk}
\begin{proof}
  First we collect some properties of the Rankin--Selberg $L$-functions for $\GL_m\times \GL_n$.
  By the Rankin--Selberg method, for an \IUCARep{} $\tau$, $L^S(s , \tau\times \tau^\vee)$ has a simple pole at $s=1$ and is non-zero holomorphic  for $\Re(s)\ge 1$ and $s\neq 1$;
  for \IUCAReps{} $\tau$ and $\tau'$ such that $\tau\not\isom\tau'$, $L^S(s , \tau \times \tau'^\vee)$  is non-zero holomorphic  for $\Re(s)\ge 1$. These results can be found in Cogdell's notes \cite{cogdell:_notes_L-fun-GLn} which collect the results from \cite{JPSS-Rankin-Selberg-MR701565,Jacquet-Shalika-MR432596,Shahidi-MR498494,Shahidi-MR561534}. % [check if we have better bounds. 20220313 No.]

  Assume that  $\phi=    \boxplus_{i=1}^r  (\tau_i,b_i)$.
  Then  
  \begin{equation*}
    L^S(s,\pi\times\tau^\vee)= \prod_{i=1}^r \prod_{j=0}^{b_i-1} L^S(s - \frac{b_i-1}{2}+j, \tau_i\times \tau^\vee),
  \end{equation*}
  where $S$ is a finite set of places of $F$ outside of which all data are unramified.

  Assume that  $\phi$ has a $(\tau,b)$-factor with $b$ maximal among all factors, then by the properties of the Rankin--Selberg $L$-functions, we see that $L^S(s,\pi\times\tau^\vee)$ has a pole at $s=\frac{b+1}{2}$ and it is maximal.

  Next assume that the partial $L$-function $L^S(s,\pi\times \tau^\vee)$ has a pole at $s=\frac{b'+1}{2}$. If $\phi$ has no $(\tau,c)$-factor for any $c\in \ZZ_{>0}$, then $L^S(s,\pi\times\tau^\vee)$ is holomorphic for all $s\in\CC$ and we get a contradiction. Thus  $\phi$ has a $(\tau,b)$-factor.
  We take  $b$ maximal among all factors of the form $(\tau,*)$ in $\phi$. As $b$ may not be maximal among all simple factors of $\phi$, we can only conclude that $L^S(s,\pi\times\tau^\vee)$ is holomorphic for $\Re(s)>\frac{b+1}{2}$. Thus $b'\le b$.  
\end{proof}

Given an \ICARep{} $\pi$, write $\phi_{\pi}$ for the global $A$-parameter of $\pi$.
By studying poles of $L^S(s,\pi\times\tau^\vee)$ for varying $\tau$'s, we can detect the existence of $(\tau,b)$-factors with maximal $b$ in $\phi_{\pi}$. We would also like to construct an irreducible cuspidal automorphic representation with global $A$-parameter $\phi_{\pi} \boxminus (\tau,b)$ which means removing the  $(\tau,b)$-factor from $\phi_\pi$ if $\phi_\pi$  has a $(\tau,b)$-factor. Doing this recursively, we will be able to compute the global $A$-parameter of a given irreducible cuspidal automorphic representation. In reverse, the construction should produce concrete examples of cuspidal automorphic representations in a given global $A$-packet with an elliptic global $A$-parameter. This will be investigated in our future work.

In this article, we focus  our attention on the study of poles of $L^S(s,\pi\times\tau^\vee)$ where $\tau$ is a conjugate self-dual \ICARep{} of $\RGL_1(\AA)$. Now we write $\chi$ for $\tau$ to emphasise that we are considering the case of twisting by characters. This case  has been well-studied and it is known that the poles of $L^{S}(s,\pi\times\chi^{\vee})$ are intricately  related to invariants of theta correspondence via the Rallis inner product formula which relates the inner product of two theta lifts to a residue or a value of the $L$-function.
 We refer the readers to \cite{Kudla-Rallis-Siegel-Weil-first-term-id-MR1289491,Wu-MR3595433,GQT-MR3279536,Yamana-Lfun-MR3211043} for details. One of the key steps is the regularised Siegel--Weil formula which relates a theta integral to a residue or a value of a Siegel Eisenstein series. Our work considers an Eisenstein series which is not of Siegel type, but which is  closely  related to $L(s,\pi\times\chi^{\vee})$.
% Through the study of theta correspondance, \cite{Kudla-Rallis-Siegel-Weil-first-term-id-MR1289491} gave a list of locations when the partial $L$-function can possibly have a pole in the $\Sp$ case. 
% [ The list contains only positive values, so the theorem says that there is no negative poles. positive poles of the completed $L$-functions also only have these possible locations? Due to symmetry, the completed $L$-function has negative poles. ]

\section{Eisenstein series attached to $\chi\boxtimes\pi$}
\label{sec:eisenst-seri}

In this section we deviate slightly from the notation in Section~\ref{sec:glob-arth-param}. We use $G(X)$ to denote one of $\Sp(X)$, $\rO(X)$ and $\rU(X)$. We let $\bG(X)$ be a cover group of $G(X)$, which means $\bG(X)=\Sp(X)$ or $\Mp(X)$ if $G(X)=\Sp(X)$, $\bG(X)=\rO(X)$ if $G(X)=\rO(X)$ and $\bG(X)=\rU(X)$ if $G(X)=\rU(X)$. We adopt  similar notation to that in \cite{MW-eis-book-MR1361168}.
We define Eisenstein series on a larger group of the same type as $\bG(X)$ and collect some results on their maximal positive poles.

Let $\pi$ be an \ICARep{} of $\bG(X)(\AA)$. We always assume that $\pi$ is genuine when $\bG(X)=\Mp(X)$. Let $\chi$ be a conjugate self-dual automorphic character of $\R_{E/F}\GL_{1}(\AA)=\AA_{E}^{\times}$. When $E\neq F$, we define
\begin{equation}
  \label{eq:epsilon-chi}
  \epsilon_{\chi} =
  \begin{cases}
    0, &\quad\text{if $\chi|_{\AA_{F}^{\times}} = \mathds{1}$};\\
    1, &\quad\text{if $\chi|_{\AA_{F}^{\times}} = \varepsilon_{E/F}$}.
  \end{cases}
\end{equation}

Let $a$ be a positive integer.
Let $X_{a}$ be the $\epsilon$-skew Hermitian space over $E$ that is formed from $X$ by adjoining $a$-copies of the hyperbolic plane.
More precisely, let $\ell_{a}^{+}$ (resp. $\ell_{a}^{-}$) be a totally isotropic $a$-dimensional $E$-vector space spanned by $e_{1}^{+},\ldots , e_{a}^{+}$ (resp. $e_{1}^{-},\ldots , e_{a}^{-}$) such that $\form{e_{i}^{+}}{e_{j}^{-}}=\delta_{ij}$ where $\delta_{ij}$ is the Kronecker symbol.
Then $X_{a}=\ell_{a}^{+}\oplus X  \oplus \ell_{a}^{-}$ with $X$ orthogonal to $\ell_{a}^{+}\oplus \ell_{a}^{-}$.

Let $G(X_{a})$ be the isometry group of $X_{a}$.
Let $Q_{a}$ be the parabolic subgroup of $G(X_{a})$ that stabilises $\ell_{a}^{-}$.
Write $Q_{a}=M_{a}N_{a}$ in the Levi decomposition with $N_{a}$ being the unipotent radical and $M_{a}$ the standard Levi subgroup.
We have an isomorphism 
\begin{equation*}
  m: \RGL_{a}\times G(X)\rightarrow M_{a}.
\end{equation*}
where we identify $\RGL_{a}$ with $\RGL(\ell_{a}^{+})$.
Let $\rho_{Q_{a}}$ be the half sum of the positive roots in $N_{a}$, which can be viewed as an element in $\fa_{M_{a}}^{*}:= \Rat(M_{a})\otimes_{\ZZ}\RR$ where $\Rat(M_{a})$ is the group of rational characters of $M_{a}$.
We note that as $Q_{a}$ is a maximal parabolic subgroup, $\fa_{M_{a}}^{*}$ is one-dimensional.
Via the Shahidi normalisation \cite{Shahidi-eis-book-MR2683009}, we identify $\fa_{M_{a}}^{*}$ with $\RR$ and thus may regard $\rho_{Q_{a}}$ as the real number
\begin{align*}
  \half(\dim_{E} X +a), \quad &\text{if $G(X_{a})$ is unitary};\\
  \half(\dim_{E} X +a-1), \quad &\text{if $G(X_{a})$ is orthogonal};\\
  \half(\dim_{E} X +a+1), \quad &\text{if $G(X_{a})$ is symplectic}.
\end{align*}
Let $K_{a,v}$ be a good maximal compact subgroup of $G(X_{a})(F_{v})$ in the sense that the Iwasawa decomposition holds and set $K_{a}=\prod_{v}K_{a,v}$.

Let $\cA^{Q_{a}}(s,\chi\boxtimes\pi)$ denote the space of $\CC$-valued smooth functions $f$ on $N_{a}(\AA)M_{a}(F)\lmod G(X_{a})(\AA)$ such that
\begin{enumerate}
\item $f$ is right $K_{a}$-finite;
\item for any $x\in\RGL_{a}(\AA)$ and $g\in G(X_{a})(\AA)$ we have
  \begin{equation*}
    f(m(x,I)g) = \chi(\det(x))|\det(x)|_{\AA_{E}}^{s+\rho_{Q_{a}}} f(g);
  \end{equation*}
\item for any fixed $k\in K_{a}$, the function $h\mapsto f(m(I,h)k)$ on $G(X)(\AA)$ is in the space of $\pi$.
\end{enumerate}

Now let $\bG(X)=\Mp(X)$. This case depends on $\psi$. Let $\til{\GL}_{1}(F_{v})$ be the double cover of $\GL_{1}(F_{v})$ defined as follows. As a set it is $\GL_{1}(F_{v}) \times \mu_{2}$ and the multiplication is given by
\begin{equation*}
  (g_1,\zeta_1) (g_2, \zeta_2) = (g_1g_2,\zeta_1\zeta_2 ( g_1,  g_2)_{F_{v}})
\end{equation*}
which has a Hilbert symbol twist when multiplying the $\mu_2$-part. Analogously we define the double cover $\til{\GL}_{1}(\AA)$ of $\GL_{1}(\AA)$.
%the preimage of $\GL_{1}(F_{v})$ under the projection $\Mp(X)(F_{v})\rightarrow \Sp(X)(F_{v})$. Here $\GL_{1}(F_{v})$ is viewed as a direct factor of the
Let $\chi_{\psi,v}$ denote the genuine character of $\til{\GL}_{1}(F_{v})$ defined by
\begin{equation*}
  \chi_{\psi,v}((g,\zeta)) = \zeta \gamma_{v}(g,\psi_{\half,v})^{-1}
\end{equation*}
where $\gamma_{v}(\cdot,\psi_{\half,v})$ is a $4$-th root of unity defined via the Weil index.
It is the same  one as in \cite[page~521]{Gan-Ichino-formal-degree-MR3166215} except that we have put in the subscripts $v$.
Then 
\begin{equation*}
  \chi_{\psi}((g,\zeta)) = \zeta \prod_{v} \gamma_{v}(g_{v},\psi_{\half,v})^{-1}
\end{equation*}
is a genuine automorphic character of $\til{\GL}_{1}(\AA)$. Let $\til{K}_{a}$ denote the preimage of $K_{a}$ under the projection $\Mp(X_{a})(\AA)\rightarrow \Sp(X_{a})(\AA)$. We will also use $\til{\ }$ to denote the preimages of other subgroups of $\Sp(X_{a})(\AA)$. Let $\til{m}$ be the isomorphism 
\begin{equation*}
  \til{\GL}_{a}(\AA) \times_{\mu_{2}} \bG(X)(\AA)\rightarrow \til{M}_{a}(\AA)
\end{equation*}
that lifts $m: \GL_{a}(\AA)\times G(X)(\AA)\rightarrow M_{a}(\AA)$. Let $\til{\det}$ be the homomorphism 
\begin{align*}
  \til{\GL}_{a}(\AA) &\rightarrow \til{\GL}_{1}(\AA)\\
  (x,\zeta) &\mapsto (\det(x),\zeta).
\end{align*}
We keep writing $\det$ for the non-genuine homomorphism
\begin{align*}
  \til{\GL}_{a}(\AA) &\rightarrow \GL_{1}(\AA)\\
  (x,\zeta) &\mapsto \det(x).
\end{align*}
Given a non-genuine representation $\tau$ of $\til{\GL}_{a}(\AA)$, we can twist it by $\chi_{\psi}\circ\til{\det}$ to get a genuine representation which we denote by $\tau\chi_{\psi}$.

We remark that there are canonical embeddings of $N_{a}(\AA)$ and $\Sp(X_{a})(F)$ to $\Mp(X_{a})(\AA)$, so we may regard them as subgroups of $\bG(X_{a})(\AA)$.
Let $\cA^{Q_{a}}_{\psi}(s,\chi\boxtimes\pi)$ denote the space of $\CC$-valued smooth functions $f$ on $N_{a}(\AA)M_{a}(F)\lmod \bG(X_{a})(\AA)$ such that
\begin{enumerate}
\item $f$ is right $\til{K}_{a}$-finite;
\item for any $x\in\til{\GL}_{a}(\AA)$ and $g\in \bG(X_{a})(\AA)$ we have
  \begin{equation*}
    f(\til{m}(x,I)g) = \chi\chi_{\psi}(\til{\det}(x))|\det(x)|_{\AA_{E}}^{s+\rho_{Q_{a}}} f(g);
  \end{equation*}
\item for any fixed $k\in \til{K}_{a}$, the function $h\mapsto f(\til{m}(I,h)k)$ on $\bG(X)(\AA)$ is in the space of $\pi$.
\end{enumerate}

To unify notation, we will also write $\cA^{Q_{a}}_{\psi}(s,\chi\boxtimes\pi)$ for  $\cA^{Q_{a}}(s,\chi\boxtimes\pi)$ in the non-metaplectic case.
It should be clear from the context whether we are treating the  $\Sp(X)$ case or the $\Mp(X)$ case.

Now return to the general case, so $\bG(X)$ is one of $\Sp(X)$, $\rO(X)$, $\rU(X)$ and $\Mp(X)$.
Let $f_{s}$ be a holomorphic section of $\cA^{Q_{a}}_{\psi}(s,\chi\boxtimes\pi)$.
We associate to it the Eisenstein series
\begin{equation*}
  E_{\psi}^{Q_{a}}(g,f_{s}):=\sum_{\gamma\in Q_{a}(F)\lmod G(X_{a})(F)}f_{s}(\gamma g).
\end{equation*}
Note that the series is over $\gamma\in Q_{a}(F)\lmod \Sp(X_{a})(F)$ when $\bG(X)=\Mp(X)$.
By Langlands' theory of Eisenstein series \cite[IV.1]{MW-eis-book-MR1361168}, this series is absolutely convergent for $\Re(s)>\rho_{Q_{a}}$, has meromorphic continuation to the whole $s$-plane, its poles lie on root hyperplanes and there are only finitely many poles in the positive Weyl chamber.
By our identification of $\fa_{M_{a}}^{*}$ with $\RR$ and the fact that $\chi$ is conjugate self-dual, the statements on poles mean that the poles are all real and that  there are  finitely many poles in the half plane $\Re(s)>0$.

We give the setup for any positive integer $a$, though we will only need $a=1$ in the statements of our results. However the proofs require `going up the Witt tower' to  $\bG(X_{a})$ for $a$ large enough. Since we plan to prove analogous results for quaternionic unitary groups in the future, we keep the setup for general $a$. % We also record a relation between the maximal poles of the Eisenstein series for different $a$'s.

% \begin{prop}
%   Assume that $E_{\psi}^{Q_{a}}(g,f_{s})$ has a positive pole and let $s_{\max,a}$ be the  maximal positive pole.  Then for all $b\in\ZZ_{>0}$ and $b\ge a$, the maximal positive pole of $E_{\psi}^{Q_{b}}(g,f_{s})$ is at $s_{\max,b} = s_{\max,a} + \half(b-a)$.
% \end{prop}
% \begin{rmk}
%   This is  \cite[Proposition~2.1]{JW-sympletic-MR3805648} in the symplectic case, \cite[Proposition~3.1]{wu22:_period-metaplectic-JNT-in-press} in the metaplectic case, \cite[Proposition~2.1]{JW-unitary-MR3435720} in the unitary case and \cite[Remarque~1]{Moeglin-eis-pole-theta-MR1473165} and \cite[Proposition~2.1]{GJS-MR2540878}  in the orthogonal case.
%   It results from computing constant terms of the Eisenstein series.
% \end{rmk}

There is a relation between poles of $L$-functions and the Eisenstein series.
\begin{prop}\label{prop:L-pole-Eis-pole}
  \begin{enumerate}
  \item Assume that the partial $L$-function $L^S_{\psi}(s,\pi\times\chi^\vee)$ has its rightmost positive pole at $s=s_0$.
    Then $E_{\psi}^{Q_{1}}(g,f_s)$ has a pole at $s=s_0$.
  \item \label{item:1} Assume that the partial $L$-function $L^S_{\psi}(s,\pi\times\chi^\vee)$ is non-vanishing at $s=\half$ and is holomorphic for $\Re(s)>\half$. Assume that
    \begin{align*}
      &\text{$\bG(X)=\rU(X)$ with $\dim X \equiv \epsilon_{\chi} \pmod 2$};\\      
      &\text{$\bG(X)=\rO(X)$ with $\dim X$ odd};\\
      &\text{$\bG(X)=\Mp(X)$}.      
    \end{align*}
    Then $E_{\psi}^{Q_{1}}(g,f_s)$ has a pole at $s=\half$.
  \end{enumerate}
\end{prop}
\begin{rmk}
  This is  \cite[Proposition~2.2]{JW-sympletic-MR3805648} in the symplectic case, \cite[Proposition~3.2]{wu22:_period-metaplectic-JNT-in-press} in the metaplectic case, \cite[Proposition~2.2]{JW-unitary-MR3435720} in the unitary case and \cite[Remarque~2]{Moeglin-eis-pole-theta-MR1473165} and  \cite[Proposition~2.2]{JW-unitary-MR3435720} in the orthogonal case.
\end{rmk}
\begin{rmk}
  The allowed $\bG(X)$'s in item~\eqref{item:1} are those for which we have theta dichotomy and epsilon dichotomy (in the local non-archimedean setting). See \cite[Corollary~9.2, Theorem~11.1]{Gan-Ichino-formal-degree-MR3166215}.
\end{rmk}
\begin{rmk}
  When $\pi$ is  a cuspidal member in a generic global $A$-packet of $\bG(X)(\AA)$, there is a more precise result. See Theorem~\ref{thm:generic-param-L-pole-Eis-pole} which was proved in \cite{JLZ-MR3079762} and strengthened in \cite{JZ-automorphic-modules-MR4088351}.
\end{rmk}

% Now we go back to the setting that $\pi$ is an \ICARep{}.
We summarise the results on the maximal positive pole of $E_{\psi}^{Q_{1}}(g,f_s)$ from \cite[Theorem~3.1]{GJS-MR2540878},  \cite[Theorem~3.1]{JW-unitary-MR3435720}, \cite[Theorem~2,8]{JW-sympletic-MR3805648} and \cite[Theorem~4.2]{wu22:_period-metaplectic-JNT-in-press}.

\begin{thm}\label{thm:max-pole-eis}
  The maximal positive  pole of $E_{\psi}^{Q_{1}}(g,f_s)$ is of the form
  \begin{equation}\label{eq:eis-max-pole}
    s=
    \begin{cases}
      \half (\dim X + 1 - (2j+\epsilon_\chi)),   & \quad \text{if $\bG(X)=\rU(X)$};\\
      \half (\dim X  - 2j), & \quad \text{if $\bG(X)=\rO(X)$};\\
      \half (\dim X + 2 - 2j), & \quad \text{if $\bG(X)=\Sp(X)$};\\
      \half (\dim X + 2 - (2j+1)), & \quad \text{if $\bG(X)=\Mp(X)$}
    \end{cases}
  \end{equation}
  where $j\in\ZZ$ such that
  \begin{equation}
    \label{eq:like-dimY}
    \begin{cases}
      r_X \le 2j+\epsilon_\chi < \dim X + 1,  & \quad \text{if $\bG(X)=\rU(X)$};\\
      r_X \le 2j <\dim X, & \quad \text{if $\bG(X)=\rO(X)$};\\
      r_X \le 2j < \dim X + 2, & \quad \text{if $\bG(X)=\Sp(X)$};\\
      r_X \le 2j+1 < \dim X + 2,  & \quad \text{if $\bG(X)=\Mp(X)$}
    \end{cases}
  \end{equation}
  where $r_X$ denotes the Witt index of $X$.
\end{thm}

\begin{rmk}
  The middle quantities in the inequalities of \eqref{eq:like-dimY} are, in fact, the lowest occurrence index of $\pi$ in the global theta lift which depends on $\chi$ and $\psi$.
  (See Theorem~\ref{thm:LO-equiv-eis-max-pole}.)
  In some cases, the lowest occurrence index turns out to be independent of $\psi$. 
  \begin{rmk}\label{rmk:eis-pole-rx-le}
    To derive  the inequalities $r_{X}\le ... $, we already need to make use of properties of the global theta correspondence. The other parts of the statements can be derived by relating our Eisenstein series to Siegel Eisenstein series whose poles are completely known. We note that via the Siegel--Weil formula, Siegel Eisenstein series are related to global theta correspondence.
  \end{rmk}
  % [move this part to sec on theta? I have moved this somewhere.]
  % We may reformulate \eqref{eq:eis-max-pole} as
  % \begin{equation}\label{eq:unified-eis-max-pole}
  %   s=\half(d_{\bG(X)^\vee} - d_{\bG(Y)^\vee} +1) 
  % \end{equation}
  % where $Y$ is an $\varepsilon$-Hermitian space such that $G(X)$ and $G(Y)$ form a reductive dual pair and  $\dim Y \equiv \epsilon_\chi \pmod 2$ in the unitary case and 
  % the condition \eqref{eq:like-dimY} 
  % becomes  $\dim Y  \ge r_X$ and $s>0$.

  % As will be discussed later,  if  $Y$ is an $\varepsilon$-Hermitian space that realises the lowest occurrence of the theta lift of $\pi$, then the maximal pole of the Eisenstein series is given by \eqref{eq:unified-eis-max-pole}. 
\end{rmk}

\section{Theta correspondence}
\label{sec:theta-correspondence}

%Define FO and LO. Orthogonal case, use Sun-Zhu's enhanced tower?
We keep the notation of Section~\ref{sec:eisenst-seri}. First we define the theta lifts and the two invariants called the first occurrence index and the lowest occurrence index. Then we relate the invariants to poles of our Eisenstein series.

Recall that we have taken an $\epsilon$-skew Hermitian space $X$ over $E$.
Let $Y$ be an $\epsilon$-Hermitian space equipped with the form $\form{\ }{\ }_{Y}$. 
We note that $\form{\ }{\ }_{Y}$ is an $F$-bilinear pairing
\begin{align*}
  \form{\ }{\ }_Y : Y \times Y \rightarrow E
\end{align*}
such that
\begin{equation*}
  \form{y_{2}}{y_{1}}_Y = \epsilon\form{y_{1}}{y_{2}}_{Y}^\varrho, \quad \form{y_{1}a}{y_{2}b}_{Y} = a^\varrho \form{y_{1}}{y_{2}}_{Y}b
\end{equation*}
for all $a,b\in E$ and $y_{1},y_{2}\in Y$.
Let $G(Y)$ be its isometry group. We note that $G(X)$ acts on $X$ from the right while $G(Y)$ acts on $Y$ from the left.
Let $W$ be the vector space $\RS_{E/F}(Y\otimes_{E}X)$ over $F$ and equip it with the symplectic form
\begin{align*}
  \form{\ }{\ }_{W} : W \times W \rightarrow F
\end{align*}
given by
\begin{align*}
  \form{y_{1}\otimes x_{1}}{y_{2}\otimes x_{2}}_{W}= \tr_{E/F}(\form{y_{1}}{y_{2}}_{Y}\form{x_{1}}{x_{2}}_{X}^{\varrho}).
\end{align*}
% $\varrho$ is to take care of the forms, one being linear in first variable, the other being linear in second variable. With this, the symplectic is well-defined.
With this set-up, $G(X)$ and $G(Y)$ form a reductive dual pair inside $\Sp(W)$. Let $W=W^{+}\oplus W^{-}$ be a polarisation of $W$. Let $\Mp^{\CC^{1}}(W)(F_{v})$ be the $\CC^{1}$-metaplectic extension of $\Sp(W)(F_{v})$. % Note that when $v$ is complex,  the extension splits. 
Let $\omega_{v}$ denote the Weil representation of $\Mp^{\CC^{1}}(W)(F_{v})$ realised on the space of Schwartz functions $\cS(W^{+}(F_{v}))$. The Weil representation depends on the additive character $\psi_{v}$, but we suppress it from notation. When $v$ is archimedean, we actually take the Fock model \cite{Howe-MR985172} rather than the full Schwartz space and it is a $(\mathfrak{sp}(W)(F_{v}),\til{K}_{\Sp(W),v})$-module but  we abuse language and call it a representation of $\Mp^{\CC^{1}}(W)(F_{v})$.  When neither $G(X)$ or $G(Y)$ is an odd orthogonal group, by \cite{Kudla-MR1286835} there exists a homomorphism 
\begin{equation*}
  G(X)(F_{v})\times G(Y)(F_{v})\rightarrow\Mp^{\CC^{1}}(W)(F_{v})
\end{equation*}
that lifts the obvious map $G(X)(F_{v})\times G(Y)(F_{v})\rightarrow\Sp(W)(F_{v})$. In this case, set $\bG(X)=G(X)$ (resp. $\bG(Y)=G(Y)$). When $G(X)$ is an odd orthogonal group, we take $\bG(Y)(F_{v})$ to be the metaplectic double cover of $G(Y)(F_{v})$ and set $\bG(X)=G(X)$.  Then by \cite{Kudla-MR1286835} there exists a homomorphism 
\begin{equation*}
  G(X)(F_{v})\times \bG(Y)(F_{v})\rightarrow\Mp^{\CC^{1}}(W)(F_{v})
\end{equation*}
that lifts  $G(X)(F_{v})\times G(Y)(F_{v})\rightarrow\Sp(W)(F_{v})$. The case is analogous when $G(Y)$ is an odd orthogonal group.  In any case, we get a homomorphism 
\begin{equation*}
  \iota_{v}:\bG(X)(F_{v})\times \bG(Y)(F_{v})\rightarrow \Mp^{\CC^{1}}(F_{v}).
\end{equation*}
It should be clear from the context when $\bG(X)$ (resp.
$\bG(Y)$) refers to a cover group and when it is not truly a cover.
In the unitary case, there are many choices of $\iota_{v}$.
Once we fix $\chi$ and an additional character $\chi_{2}$, then $\iota_{v}$ is fixed.
This is worked out in great details in \cite[Section~1]{Harris-Kudla-Sweet-MR1327161}. Our $(\chi,\chi_{2})$ matches  $(\chi_{1},\chi_{2})$ in \cite[(0.2)]{Harris-Kudla-Sweet-MR1327161}.
We note that $Y$ should be compatible with $\chi$ and $\chi$ determines the embedding of $G(X)(\AA)$ into $\Mp^{\CC^{1}}(\AA)$ whereas $X$ should be compatible with $\chi_{2}$ and $\chi_{2}$ determines the embedding of $G(Y)(\AA)$ into $\Mp^{\CC^{1}}(\AA)$.
By `compatible', we mean $\epsilon_{\chi} \equiv \dim Y \pmod 2$ (resp.
$\epsilon_{\chi_{2}} \equiv \dim X \pmod 2$).
See \cite{Kudla-MR1286835} for more details.
% We will see below that the choice of $\chi_{2}$ is not essential as long as it is compatible with $X$.
We pull back $\omega_{v}$ to $\bG(X)(F_{v})\times \bG(Y)(F_{v})$ via $\iota_{v}$ and still denote the representation by $\omega_{v}$.

Denote by $\iota$ the adelic analogue of $\iota_{v}$. We also have the (global) Weil representation $\omega$ of $\Mp^{\CC^{1}}(\AA)$  on the Schwartz space $\cS(W^{+}(\AA))$ and its pullback via $\iota$ to $\bG(X)(\AA)\times \bG(Y)(\AA)$.

Then we can define the theta function which will be used as a kernel function. Let
\begin{equation*}
  \theta_{X,Y}(g,h,\Phi) := \sum_{w\in W^{+}(F)} \omega(\iota(g,h))\Phi(w)
\end{equation*}
for $g\in \bG(X)(\AA)$, $h\in \bG(Y)(\AA)$ and $\Phi\in\cS(W^{+}(\AA))$. It is absolutely convergent and is an automorphic form on $\bG(X)(\AA)\times \bG(Y)(\AA)$.
For $f\in\pi$, set
\begin{equation*}
  \theta_{X}^{Y}(f,\Phi) := \int_{[\bG(X)]} \overline{f(g)}\theta_{X,Y}(g,h,\Phi) dg.
\end{equation*}
Note that we write $[\bG(X)]$ for $G(X)(F)\lmod G(X)(\AA)$ when $\bG(X)$ is not metaplectic and $G(X)(F)\lmod \bG(X)(\AA)$ or more explicitly $\Sp(X)(F)\lmod \Mp(X)(\AA)$ when $\bG(X)$ is metaplectic.
This is an automorphic form on $\bG(Y)(\AA)$. It depends on $\chi$ and $\chi_{2}$ in the unitary case and when we want to emphasise the dependency, we will write $\theta_{X,\chi_{2}}^{Y,\chi}(f,\Phi)$.
Let $\Theta_{X}^{Y}(\pi)$ denote the space of functions spanned by $\theta_{X}^{Y}(f,\Phi)$'s and let $\Theta_{X,,\chi_{2}}^{Y,\chi}(\pi)$ denote the space of functions spanned by $\theta_{X,\chi_{2}}^{Y,\chi}(f,\Phi)$'s in the unitary case. 

From now on assume that $Y$ is anisotropic (possibly zero), so that it sits at the bottom of its Witt tower. Define $Y_{r}$ to be the $\epsilon$-Hermitian space formed by adjoining $r$-copies of the hyperbolic plane to $Y$. These $Y_{r}$'s form the Witt tower of $Y$. By the tower property \cite{Rallis-MR743016,Wu-irred-theta-MR3071813}, if the theta lift to $\bG(Y_{r})$ is non-zero then the theta lift to $\bG(Y_{r'})$ is also non-zero for all $r'\ge r$.

Define the first occurrence index of $\pi$ in the Witt tower of $Y$ to be 
\begin{equation}\label{eq:FO-defn}
  \FO_{X}^{Y,\chi}(\pi) :=
  \begin{cases}
    \min \left\{ \dim Y_{r} \middle| \Theta_{X,\chi_{2}}^{Y_{r},\chi}(\pi) \neq 0 \right\}, \quad &\text{if $\bG(X)=\rU(X)$};\\    
    \min \left\{ \dim Y_{r} \middle| \Theta_{X}^{Y_{r}}(\pi\otimes(\chi\circ\upsilon)) \neq 0 \right\}, \quad &\text{if $\bG(X)=\rO(X)$};\\    
    \min \left\{ \dim Y_{r} \middle| \Theta_{X}^{Y_{r}}(\pi) \neq 0 \right\}, \quad &\text{if $\bG(X)=\Sp(X)$ or $\Mp(X)$}.
  \end{cases}
\end{equation}
Note that it depends on $\chi$ but not on $\chi_{2}$ in the unitary case as changing $\chi_{2}$ to another compatible one produces only a character twist on $\Theta_{X,\chi_{2}}^{Y_{r},\chi}(\pi)$. For more details, see \cite[(1-1)]{wu22:_poles_eisen-unitary-PJM}. In the orthogonal case, we twist $\pi$ by $\chi\circ\upsilon$ where $\upsilon$ denotes the spinor norm. If $\bG(X)=\Sp(X)$ or $\Mp(X)$, we require that $\chi_{Y} = \chi$ where $\chi_{Y}$ is the quadratic automorphic character of $\GL_{1}(\AA)$ associated to $Y$ given by
\begin{equation*}
  \chi_Y (g) = (g,(-1)^{\dim Y (\dim Y -1)/2}\det\form
{\ } {\ }_Y),
\end{equation*}
where $(\ ,\ )$ is the Hilbert symbol.

Define the lowest occurrence index to be
\begin{equation}\label{eq:defn-LO}
  \LO_{X}^{\chi}(\pi) := \min \left\{ \FO_{X}^{Y,\chi}(\pi) \middle| \ \text{$Y$ is compatible with $\chi$}   \right\},
\end{equation}
when $\bG(X)=\rU(X), \Sp(X)$ or $\Mp(X)$. Here
compatibility means that
\begin{align}
 &\dim Y \equiv \epsilon_{\chi} \pmod 2, &\text{ if } \bG(X)=\rU(X);\nonumber\\ 
 &\chi_{Y} = \chi, &\text{ if } \bG(X)=\Sp(X) \text{ or } \Mp(X).\label{eq:theta-chi-Y-compatible-cond}
\end{align}
Define the lowest occurrence index to be
\begin{equation}\label{eq:LO-defn-O}
  \LO_{X}^{\chi}(\pi) := \min \left\{ \FO_{X}^{Y,\chi}(\pi\otimes \sgn_{T}) \middle| \  \text{$T$ a set of even number of places of $F$}  \right\},
\end{equation}
when $\bG(X)=\rO(X)$.

We have the following relations of the first occurrence (resp. the lowest occurrence) and the poles (resp. the maximal positive pole) of the Eisenstein series.  See  \cite[Corollary~3.9, Theorem~3.10]{JW-sympletic-MR3805648} for the symplectic case, \cite[Corollary~6.3, Theorem~6.4]{wu22:_period-metaplectic-JNT-in-press} for the metaplectic case, \cite[Corollary~3.5, Corollary~3.7]{JW-unitary-MR3435720} for the unitary case and \cite[Theorem~5.1, Theorem~1.3]{GJS-MR2540878} for the orthogonal case.

\begin{thm}\label{thm:FO-eis-pole}
  Let $\pi$ be an \ICARep{} of $\bG(X)(\AA)$. Let $\chi$ be a conjugate self-dual automorphic character of $\R_{E/F}\GL_{1}(\AA)$. Let $Y$ be an anisotropic $\epsilon$-Hermitian space that is compatible with $\chi$ in the sense of \eqref{eq:theta-chi-Y-compatible-cond}. Assume that $\FO_{X}^{Y,\chi}(\pi)=\dim Y +2r$.  Set
  \begin{equation}\label{eq:eis-pole-is-diff-of-sizes-of-two-groups}
    s_{0} =
    \begin{cases}
      \half(\dim X +1 - (\dim Y+2r)), \quad &\text{if $\bG(X)=\rU(X)$};\\
      \half(\dim X - (\dim Y+2r) ), \quad &\text{if $\bG(X)=\rO(X)$};\\
      \half(\dim X +2 - (\dim Y+2r) ), \quad &\text{if $\bG(X)=\Sp(X)$ or $\Mp(X)$}.
    \end{cases}
  \end{equation}
  Assume that $s_{0}\neq 0$. If $\bG(X)=\rO(X)$ and $s_{0}<0$, further assume that $\half\dim X <r<\dim X-2$.
  Then $s=s_{0}$ is a pole of the Eisenstein series $E_{\psi}^{Q_{1}}(g,f_{s})$ for some choice of $f_{s}\in \cA_{\psi}^{Q_{1}}(s,\chi\boxtimes\pi)$.
\end{thm}
\begin{rmk}
  Using the notation from Section~\ref{sec:glob-arth-param}. The quantity $s_{0}$ in \eqref{eq:eis-pole-is-diff-of-sizes-of-two-groups}   can be written uniformly as
  \begin{equation*}
    \half(d_{\bG(X)^\vee} - d_{\bG(Y_{r})^\vee} +1).
  \end{equation*}
\end{rmk}
\begin{rmk}
  Note that we always have  $r\le \dim X$. The extra condition when $\bG(X)=\rO(X)$ is to avoid treating period integrals over the orthogonal groups of split binary quadratic forms, as our methods cannot deal with the technicality.
  Theorem~\ref{thm:FO-eis-pole} allows negative $s_{0}$. It is possible to detect non-maximal poles and  negative poles of the Eisenstein series by the first occurrence indices.
\end{rmk}
% just too difficult to state using Y.
% \begin{thm}
%   Let $\pi$ be an \ICARep{}. Assume that $\LO_{X}^{\chi}(\pi)$ is achieved by $Y$. Let  $s_{0}=\half(d_{\bG(X)^\vee} - d_{\bG(Y)^\vee} +1)$. Assume that $s_{0}>0$. Then 
%   The maximal positive pole of $E(g,f_s)$ is at $s=s_{0}$.

%   Assume that the maximal positive pole of $E(g,f_s)$ is at $s=\half(d_{G(X)^\vee} - d +1)$ for some integer $d$. Then there exists an $\epsilon$-Hermitian space $Y$ such that $Y$ compatible with $\chi$, $d_{\bG(Y)^\vee} = d$ and $\LO_{X}^{\chi}(\pi)=\dim Y$.
% \end{thm}

\begin{thm}\label{thm:LO-equiv-eis-max-pole}
  Let $\pi$ be an \ICARep{}  of $\bG(X)(\AA)$.
Let $\chi$ be a conjugate self-dual automorphic character of $\R_{E/F}\GL_{1}(\AA)$.
Then the maximal positive pole of $E_{\psi}^{Q_{1}}(g,f_s)$ for $f_{s}$ running over $\cA_{\psi}^{Q_{1}}(s,\chi\boxtimes\pi)$ is at $s=s_{0} \in \RR$ if and only if
  \begin{equation}\label{eq:LO-in-terms-of-Eis-pole}
    \LO_{X}^{\chi}(\pi) =
    \begin{cases}
      \dim X+ 1 - 2s_{0},  \quad &\text{if $\bG(X)=\rU(X)$};\\
      \dim X - 2s_{0},  \quad &\text{if $\bG(X)=\rO(X)$};\\
      \dim X +2 - 2s_{0}, \quad &\text{if $\bG(X)=\Sp(X)$ or $\Mp(X)$}.
    \end{cases}
  \end{equation}
\end{thm}
\begin{rmk}
  Theorem~\ref{thm:LO-equiv-eis-max-pole} does not allow negative $s_{0}$.
\end{rmk}

% idea of proofs. commented out
% Theorem~\ref{thm:FO-eis-pole} is deduced from the computation of period integrals involving a truncated Eisenstein series and a theta kernel.  Roughly, the period integrals can be unfolded to a sum of several other period integrals, most of which vanish by the condition on the first occurrence. Those that do not vanish produce a pole at the prescribed location $s_{0}$ in \eqref{eq:eis-pole-is-diff-of-sizes-of-two-groups}. 

% Theorem~\ref{thm:LO-equiv-eis-max-pole} goes in two directions. Given the lowest occurrence index, $\dim Y+2r$, Theorem~\ref{thm:FO-eis-pole} already shows that the maximal pole of the Eisenstein series is bigger than or equal to the quantity on the right in \eqref{eq:eis-pole-is-diff-of-sizes-of-two-groups}.  On the other hand, by first relating our Eisenstein series to Siegel Eisenstein series and then  using the Siegel--Weil formula to relate to global theta lifts, we find that the maximal pole of our Eisenstein series is smaller than or equal to the quantity on the right in \eqref{eq:eis-pole-is-diff-of-sizes-of-two-groups}.

In Remark~\ref{rmk:eis-pole-rx-le}, we mentioned that the part $r_{X}\le ...$ in Theorem~\ref{thm:max-pole-eis} is proved by using theta correspondence. What we used is that we always have $ \LO_{X}^{\chi}(\pi)\ge r_{X}$ by the stable range condition \cite[Theorem~I.2.1]{Rallis-MR743016}.

\section{Application to Global Arthur Packets}
\label{sec:appl-glob-arth}

We have derived relations among $(\chi,b)$-factors of global $A$-parameters, poles of partial $L$-functions, poles of Eisenstein series and lowest occurrence indices of global theta lifts. Combining these, we have the following  implication on global $A$-packets.

\begin{thm}
  Let $\pi$ be an \ICARep{} of $\bG(X)(\AA)$. Let $\phi_\pi$ be its global $A$-parameter. Let $\chi$ be a conjugate self-dual automorphic character of $\R_{E/F}\GL_{1}(\AA)$.  Assume that $\phi_\pi$ has a $(\chi,b)$-factor for some positive integer $b$. Then 
  \begin{equation}\label{eq:bound-on-b}
    b \le
    \begin{cases}
      \dim X - r_X,  & \quad \text{if $\bG(X)=\rU(X)$};\\
      \dim X  - r_X -1, & \quad \text{if $\bG(X)=\rO(X)$};\\
      \dim X - r_{X} +1 = \half \dim X + 1, & \quad \text{if $\bG(X)=\Sp(X)$ or $\Mp(X)$}.
    \end{cases}
  \end{equation}
  where $r_{X}$ denotes the Witt index of $X$.
\end{thm}
\begin{proof}
  If $b$ is not maximal among all factors $(\tau,b)$ appearing in $\phi_\pi$, then $b<\half d_{\bG(X)^\vee}$.
Then it is clear that $b$ satisfies \eqref{eq:bound-on-b}.
Now we assume that $b$ is maximal among all factors appearing in $\phi_\pi$.
By Prop.~\ref{prop:tau-b-L-pole}, $L^S(s,\pi\times \chi^{-1})$ has its rightmost pole at $s=\half(b+1)$.
% Assume $s=\half(b_1+1)$ is its rightmost pole with $b_1\ge b$.
Then by Prop.~\ref{prop:L-pole-Eis-pole}, $E_{\psi}^{Q_{1}}(g,f_s)$ has a pole at $s=\half(b+1)$ for some choice of $f_{s}$.
Assume that $s=\half(b_{1}+1)$ is the rightmost pole of the Eisenstein series with $b_{1}\ge b$.
By Thm.~\ref{thm:max-pole-eis},
  \begin{equation*}
    \half(b_{1}+1) \le
    \begin{cases}
      \half (\dim X + 1 - r_X),   & \quad \text{if $\bG(X)=\rU(X)$};\\
      \half (\dim X  - r_X), & \quad \text{if $\bG(X)=\rO(X)$};\\
      \half (\dim X + 2 - r_X), & \quad \text{if $\bG(X)=\Sp(X)$ or $\Mp(X)$}.
    \end{cases}
  \end{equation*}
  or in other words, $b_1$ is less than or equal to the quantity on the RHS of \eqref{eq:bound-on-b}.
  Using the fact that $b\le b_1$, we get the desired bound for $b$.
\end{proof}

\begin{rmk}
  Our result generalises \cite[Theorem~3.1]{Jiang-Liu-MR3969876} for symplectic groups to classical groups and metaplectic groups. In addition, we do not require the  assumption on the wave front set in \cite[Theorem~3.1]{Jiang-Liu-MR3969876}. 
  This type of result has been used in \cite[Section~5]{Jiang-Liu-MR3969876} to find a Ramanujan bound  which measures the departure of the local components of a cuspidal $\pi$ from being tempered.

  The metaplectic case has been treated in \cite[Theorem~0.1]{wu22:_period-metaplectic-JNT}, though the proof is not written down explicitly.
Here we supply the detailed arguments for all classical groups and metaplectic groups uniformly.
\end{rmk}

The corollary below follows immediately from the theorem.
\begin{cor}\label{cor:A-packet-no-cusp}
  The global $A$-packet $\Pi_\phi$ attached to the elliptic global $A$-parameter $\phi$ cannot have a cuspidal member if $\phi$ has a $(\chi,b)$-factor with
  \begin{equation*}
    b >
    \begin{cases}
      \dim_E X - r_X,  & \quad \text{if $\bG(X)=\rU(X)$};\\
      \dim_F X  - r_X -1, & \quad \text{if $\bG(X)=\rO(X)$};\\
      \dim X - r_{X} +1 = \half \dim X + 1, & \quad \text{if $\bG(X)=\Sp(X)$ or $\Mp(X)$}.
    \end{cases}
  \end{equation*}
\end{cor}

\section{Generic Global $A$-packets}
\label{sec:gener-repr}

Following the terminology of \cite{Arthur-book-MR3135650}, we say that an elliptic global $A$-parameter is generic if it is of the form $\phi = \boxplus_{i=1}^{r} (\tau_{i},1)$ and we say a global $A$-packet is generic if its global $A$-parameter is generic. 
%Then generic  global $A$-packets are those whose global $A$-parameters are generic.
Assume that $\pi$ is a cuspidal member in a generic  global $A$-packet.
Then our results can be made more precise.
We note that our results for $\Mp(X)$ are conditional on results on normalised intertwining operators.
See Assumption~\ref{ass:normalised-intertwining-op} and Remark~\ref{rmk:metaplectic-conditional}.

First assume that $G(X)$ is quasi-split and that $\pi$ is globally generic.
 We explain what we mean by globally generic.
We use the same set-up as in \cite[Section~3]{Shahidi-MR942520}.
Let $B$ be a Borel subgroup of $G(X)$.
Let $N$ denote its unipotent radical and let $T$ be a fixed choice of   Levi subgroup of $B$.
Of course, in this case $T$ is a maximal torus of $G(X)$.
We require that $T$ is maximally split.
Let $\bar{F}$ denote an algebraic closure of $F$.
Let $\Delta$ denote the set of simple roots of $T(\bar{F})$ in $N(\bar{F})$.
Let $\{X_{\alpha}\}_{\alpha\in\Delta}$ be a $\Gal(\bar{F}/F)$-invariant set of root vectors. % and we call it an $F$-splitting, following \cite{CKPSS-MR2075885}.
Recall that $\psi$ is a fixed non-trivial automorphic  character of $\AA_{F}$ which is used in the definitions of the Weil representation and the global $A$-packets for $\Mp(X)$.
% Thus the first occurrence indices and lowerest occurrence indices depend on $\psi$, but the notation suppresses $\psi$.
It gives rise to generic characters of $N(\AA)$.
We use the one defined as follows.
For each place $v$ of $F$, we define a character $ \psi_{N,v} $ of $N(F_{v})$.
Write an element of $N(F_{v})$ as $\prod_{\alpha\in\Delta} \exp(x_{\alpha}X_{\alpha})$ for  $x_{\alpha}\in \bar{F}_{v}$ such that $\sigma x_{\alpha} = x_{\sigma\alpha}$ with $\sigma\in\Gal(\bar{F}/F)$.
Set
\begin{align*}
  \psi_{N,v}(\prod_{\alpha\in\Delta} \exp(x_{\alpha}X_{\alpha})) = \psi_{v} (\sum_{\alpha\in \Delta} x_{\alpha}).
\end{align*}
Let  $\psi_{N} = \otimes_{v}\psi_{N,v}$.
In the $\Mp(X)$ case, we view $N(\AA)$ as a subgroup of $\Mp(X)(\AA)$ via the canonical splitting.
% The one we use is the same one as in \cite[Chapter~10, page 289 and page 298]{GRS-descent-book-MR2848523} where it is denoted by $\psi_{N_{G}}$.
% More explicitly, $\psi_{N}$ is constructed   as in \cite[Section~12]{GGP-MR3202556}. 
%See $\nu$ on page 48, $\nu_{\psi}$ on page 54 and $\nu_{\psi_{0}}$ on page 55 of \cite{GGP-MR3202556} for various cases.
We require  that $\pi$ is globally generic with respect to the generic character $\psi_{N}$ of $N(\AA)$.
Thus the notion of global genericity depends on the choice of the generic automorphic character of $N(\AA)$.
However by  \cite[Appendix~A]{CKPSS-MR2075885}, the choice has no effect on the $L$-factors, the $\varepsilon$-factors  and the global $A$-parameter  for $\pi$ in the case of $\bG(X)=\Sp(X), \rO(X), \rU(X)$.
The case of $\Mp(X)$ is highly dependent on the choice.

When $\pi$ is globally generic,  $b=1$ for every factor $(\tau,b)$ in the global $A$-parameter $\phi_{\pi}$.
This is because  the Langlands functorial lift of $\pi$ is an isobaric sum of conjugate self-dual cuspidal representations of some $\RGL_{n}(\AA)$. See Theorem~11.2 of \cite{GRS-descent-book-MR2848523}.
% We do not know if our results can be generalised to all automorphic representations in generic global $A$-packets (i.e., those whose $A$-parameter is of the form $\phi = \boxplus (\tau_{i},1)$), since  the results that we need on $L$-functions and Eisenstein series in \cite{JLZ-MR3079762} are proved for  globally generic members only.
% According to the introduction of  \cite{JLZ-MR3079762}, the main obstruction to the generalising their results is the lack of normalisation of intertwining operators with the desired analytic properties at the archimedean places.
% TODO: Jiang Zhang automorphic modules. Prop 5.2 works? Take $\tau$ to be $\chi$ (just one factor)

By \cite{JLZ-MR3079762},  there is a more precise relation on the poles of $L$-functions and the poles of Eisenstein series.
The set of possible poles of the normalised Eisenstein series is determined by the complete $L$-function $L(s,\pi\times\chi^{\vee})$.
From the assumption that $\pi$ is globally generic,  in the right half plane, $L(s,\pi\times\chi^{\vee})$ has at most a simple pole at $s=1$.
In fact we only need \cite[Proposition~4.1]{JLZ-MR3079762} rather than the full strength of \cite[Theorem~1.2]{JLZ-MR3079762} which allows the induction datum to be a Speh representation on the general linear group factor of the Levi.
By \cite[Theorem~5.1]{JZ-automorphic-modules-MR4088351}, \cite[Proposition~4.1]{JLZ-MR3079762} can be strengthened to include the case where $\pi$ is a cuspidal member in a generic global $A$-packet of $\bG(X)(\AA)$ where $\bG(X) = \Sp(X), \rO(X),\rU(X)$ does not have to be quasi-split. 
We rephrase \cite[Proposition~4.1]{JLZ-MR3079762} in our context as Theorem~\ref{thm:generic-param-L-pole-Eis-pole}.

First we set up some notation and outline the method for extending \cite[Proposition~4.1]{JLZ-MR3079762}  to the case of  $\Mp(X)$.
Let
\begin{equation}\label{eq:rho+}
  \rho^{+}:=
  \begin{cases}
    \Asai^{\eta} \text{, where $\eta=(-1)^{\dim X+1}$},\quad & \begin{array}{l}
                                                                \text{if $\bG(X)=\rU(X)$,}
                                                              \end{array}\\
    \wedge^{2},  \quad & \begin{array}{l}
                          \text{if $\bG(X)=\rO(X)$ with $\dim X$ odd} \\
                          \text{ or if $\bG(X)=\Mp(X)$};
                        \end{array}\\
    \Sym^{2},  \quad & \begin{array}{l}
                         \text{if $\bG(X)=\rO(X)$ with $\dim X$ even}\\
                         \text{ or if $\bG(X)=\Sp(X)$}
                       \end{array}
  \end{cases}
\end{equation}
and
\begin{equation}\label{eq:rho-}
  \rho^{-}:=
  \begin{cases}
    \Asai^{-\eta} \text{, where $\eta=(-1)^{\dim X+1}$},\quad &\begin{array}{l}
                                                                \text{if $\bG(X)=\rU(X)$,}
                                                              \end{array}\\
    \Sym^{2}, \quad &\begin{array}{l}
                          \text{if $\bG(X)=\rO(X)$ with $\dim X$ odd} \\
                          \text{ or if $\bG(X)=\Mp(X)$};
                        \end{array}\\
    \wedge^{2},  \quad &\begin{array}{l}
                         \text{if $\bG(X)=\rO(X)$ with $\dim X$ even}\\
                         \text{ or if $\bG(X)=\Sp(X)$}.
                       \end{array}
  \end{cases}
\end{equation}

% \begin{rmk}\label{rmk:metaplectic-conditional}
  The  results of \cite{JLZ-MR3079762} do not cover the metaplectic case, but the method should generalise without difficulty.
  We explain the strategy.
  First the poles of the Eisenstein series are related to those of the intertwining operators
\begin{equation*}
  M(w_{0}, \tau|\ |^{s}\boxtimes \pi):    \Ind_{Q_{a}(\AA)}^{\bG(X_{a})(\AA)}(\tau|\ |^{s}\boxtimes\pi) \rightarrow   \Ind_{Q_{a}(\AA)}^{\bG(X_{1})(\AA)}(\tau|\ |^{-s}\boxtimes\pi)
\end{equation*}
where $\tau$ is a conjugate self-dual cuspidal automorphic representation of $\GL_{a} (\AA_E)$  % of parity $(-1)^{N_{\check{\bG(X)}}+1}$
and $w_{0}$ is the longest Weyl element in $Q_{a} \lmod G(X_{a}) / Q_{a}$.
Then define  the normalised intertwining operator 
\begin{align}\label{eq:normalised-intertwining-op}
  N(w_{0}, \tau|\ |^{s}\boxtimes \pi) :=  &\frac{L(s,\pi\times \tau^{\vee}) L(2s,\tau, \rho^{-})}{   L(s+1,\pi\times \tau^{\vee}) L(2s+1,\tau, \rho^{-}) \varepsilon(s,\pi\times \tau^{\vee}) \varepsilon(2s,\tau, \rho^{-})} \nonumber \\ 
&\quad  \cdot M(w_{0}, \tau|\ |^{s}\boxtimes \pi). 
\end{align}
    The proof of \cite[Proposition~4.1]{JLZ-MR3079762} relies on 
    the   key result that the normalised intertwining operator 
    is holomorphic and non-zero for $\Re s \ge \half$. 
 Then it boils down to finding the poles of the normalising factors or equivalently
 \begin{align*}
   \frac{L(s,\pi\times \tau^{\vee}) L(2s,\tau, \rho^{-})}{   L(s+1,\pi\times \tau^{\vee}) L(2s+1,\tau, \rho^{-})}.
 \end{align*}
  Once we have the key result available, we expect to have a version of \cite[Proposition~4.1]{JLZ-MR3079762} for the metaplectic groups.
 Note that our $\rho^{\pm}$ defined in~\eqref{eq:rho+} and~\eqref{eq:rho-} is different from the $\rho$ and $\rho^{-}$ in  \cite{JLZ-MR3079762}.
 
Then by using an inductive formula, we expect to be able to  prove \cite[Theorem~1.2]{JLZ-MR3079762} as well.
We hope to supply the details in a future work.

% \end{rmk}

Next we allow $G(X)$ to be non-quasi-split. We assume that $\pi$ is a cuspidal member in a generic global $A$-packet of $\bG(X)$.
Then by \cite[Theorem~5.1]{JZ-automorphic-modules-MR4088351}, \eqref{eq:normalised-intertwining-op} is holomorphic and non-zero for $\Re s\ge\half$ when $\bG(X) = \Sp(X), \rO(X),\rU(X)$.
Then the proof of \cite[Proposition~4.1]{JLZ-MR3079762} goes through verbatim for such $\pi$.
The proof of \cite[Theorem~5.1]{JZ-automorphic-modules-MR4088351} does not generalise readily to the case of $\Mp(X)$ as
the relevant results for $\Mp(X)$ are not available.
% proves Theorem~\ref{thm:generic-param-L-pole-Eis-pole} verbatim.

Thus we make an assumption on the normalised intertwining operator:
\begin{ass}\label{ass:normalised-intertwining-op}
  The normalised intertwining operator $N(w_{0}, \chi|\ |^{s}\boxtimes \pi )$ is holomorphic and non-zero for $\Re s \ge \half$.
\end{ass}
\begin{rmk}\label{rmk:metaplectic-conditional}
  This is shown to be true by \cite[Theorem~5.1]{JZ-automorphic-modules-MR4088351} when $\pi$ is a cuspidal member in a generic global $A$-packet of $\bG(X)(\AA)$ for $\bG(X) = \Sp(X), \rO(X),\rU(X)$.
  Thus this is only a condition when $\bG(X)=\Mp(X)$.
\end{rmk}

% \begin{thm}\label{thm:generic-L-pole-Eis-pole}
%   Let $\pi$ be an \IGCARep{} of $\bG(X)(\AA)$. Let $\chi$ be a conjugate self-dual automorphic character of $\R_{E/F}\GL_{1}(\AA)$.
%   \begin{enumerate}
%   \item Assume $\bG(X)=\rU(X)$ with $\epsilon_{\chi}\not\equiv \dim X \pmod 2$, $\rO(X)$ with $\dim X$ even or $\Sp(X)$. Then $L(s,\pi\times\chi^{\vee})$ has a pole at $s=1$ if and only if $E^{Q_{1}}(g,f_{s})$ has a pole at $s=1$ and it is its maximal pole.
%   \item Assume $\bG(X)=\rU(X)$ with $\epsilon_{\chi}\equiv \dim X \pmod 2$, $\rO(X)$ with $\dim X$ odd or $\Mp(X)$. Then   $L(s,\pi\times\chi^{\vee})$ is non-vanishing at  $s=\half$ if and only if $E_{\psi}^{Q_{1}}(g,f_{s})$ has a pole at $s=\half$ and it is its maximal pole.
%   \end{enumerate}
% \end{thm}

%%%%

\begin{thm}\label{thm:generic-param-L-pole-Eis-pole}
  Assume Assumption~\ref{ass:normalised-intertwining-op}.
  Let $\pi$ be a cuspidal member in a generic global $A$-packet of $\bG(X)(\AA)$. Let $\chi$ be a conjugate self-dual automorphic character of $\R_{E/F}\GL_{1}(\AA)$.
  \begin{enumerate}
  \item Assume $\bG(X)=\rU(X)$ with $\epsilon_{\chi}\not\equiv \dim X \pmod 2$, $\rO(X)$ with $\dim X$ even or $\Sp(X)$. Then $L(s,\pi\times\chi^{\vee})$ has a pole at $s=1$ if and only if $E^{Q_{1}}(g,f_{s})$ has a pole at $s=1$ and it is its maximal pole.
  \item Assume $\bG(X)=\rU(X)$ with $\epsilon_{\chi}\equiv \dim X \pmod 2$, $\rO(X)$ with $\dim X$ odd or $\Mp(X)$. Then   $L(s,\pi\times\chi^{\vee})$ is non-vanishing at  $s=\half$ if and only if $E_{\psi}^{Q_{1}}(g,f_{s})$ has a pole at $s=\half$ and it is its maximal pole.
  \end{enumerate}
\end{thm}

\begin{rmk}
  The result of \cite{JLZ-MR3079762} involves normalised Eisenstein series, but the normalisation has no impact on the positive poles. The following remarks use the notation in \cite{JLZ-MR3079762}. We only need the case $b=1$ in \cite{JLZ-MR3079762} which is Proposition~4.1 there. Furthermore we only apply it in the case where $\tau$ is a character. The condition that $L(s,\tau,\rho)$ has a pole at $s=1$ is automatically satisfied by the requirement on our $\chi$ that it is conjugate self-dual of parity $(-1)^{N_{\bG(X)^{\vee}}+1}$. See Section~\ref{sec:glob-arth-param}, especially Remark~\ref{rmk:conj-self-dual-and-L-fun-characterisation}.
  % [partial L complete L. Do We need GRC to get that local factors are tempered and therefore no impact on the positive poles of partial L and L. Yamana has shown local L is hol for tempered in $\Re s>0$, so our generic case is OK.]
%  [check if L should be $L(s,\pi\times\chi)$ or $L(s,\pi\times\chi^{\vee})$ in the unitary case. Maybe it is just convention. We actually use the same L. Yes, it is just a convention difference.]
\end{rmk}

% \begin{rmk}\label{rmk:metaplectic-conditional-2}
%   We note again that the case of $\Mp(X)$ is conditional on the analytic properties of  \eqref{eq:normalised-intertwining-op} as pointed out in Remark~\ref{rmk:metaplectic-conditional}.
% \end{rmk}

The global $A$-parameter $\phi_{\pi}$ can possibly have a $(\chi,1)$-factor only when $\chi$ satisfies the condition that $L(s,\chi,\rho^{+})$ has a pole at $s=1$. % We recall that $\rho^{+}$ is defined in \eqref{eq:rho+}.
Due to the parity condition on factors of an elliptic global $A$-parameter, in some cases, $\phi_{\pi}$ cannot have a $(\chi,1)$-factor.

Combining with our result (Theorem~\ref{thm:LO-equiv-eis-max-pole}) on poles of Eisenstein series and lowest occurrence indices with Theorem~\ref{thm:generic-param-L-pole-Eis-pole} which gives a precise relation between poles of the complete $L$-function and those of the Eisenstein series, we get
\begin{thm}
  Assume Assumption~\ref{ass:normalised-intertwining-op}.
  Let $\pi$ be a cuspidal member in a generic global $A$-packet of $\bG(X)(\AA)$. Let $\chi$ be a conjugate self-dual automorphic character of $\R_{E/F}\GL_{1}(\AA)$. In each of the following statements, we consider only those $\bG(X)$'s that are listed.
  \begin{enumerate}
  \item \label{item:L1-pole} Assume that $L(s,\pi\times\chi^{\vee})$ has a pole at $s=1$. Then 
    \begin{equation*}
      \LO_{X}^{\chi}(\pi) =
        \begin{cases}
      \dim X - 1, \quad &\text{if $\bG(X)=\rU(X)$ and $\epsilon_{\chi}\not\equiv \dim X \pmod 2$};\\
      \dim X - 2,  \quad &\text{if $\bG(X)=\rO(X)$ with $\dim X$ even};\\
      \dim X,  \quad &\text{if $\bG(X)=\Sp(X)$}.
      \end{cases}
    \end{equation*}
  \item \label{item:L1-no-pole} Assume that $L(s,\pi\times\chi^{\vee})$ does not have a pole at $s=1$. Then 
    \begin{equation*}
      \LO_{X}^{\chi}(\pi) \ge
      \begin{cases}
        \dim X + 1, \quad &\text{if $\bG(X)=\rU(X)$ and $\epsilon_{\chi}\not\equiv \dim X \pmod 2$};\\
        \dim X,   \quad &\text{if $\bG(X)=\rO(X)$ with $\dim X$ even};\\
        \dim X +2, \quad &\text{if $\bG(X)=\Sp(X)$}.
      \end{cases}
    \end{equation*}
  \item\label{item:Lhalf-non-van}
    Assume $L(\half,\pi\times\chi^{\vee})\neq 0$. Then 
    \begin{equation*}
      \LO_{X}^{\chi}(\pi) =
        \begin{cases}
      \dim X,  \quad &\text{if $\bG(X)=\rU(X)$ and $\epsilon_{\chi}\equiv \dim X \pmod 2$};\\
      \dim X - 1,  \quad &\text{if $\bG(X)=\rO(X)$ with $\dim X$ odd};\\
      \dim X +1,  \quad &\text{if $\bG(X)=\Mp(X)$}.
      \end{cases}
    \end{equation*}
    \item \label{item:Lhalf-van} 
    Assume $L(\half,\pi\times\chi^{\vee})=0$. Then 
    \begin{equation*}
      \LO_{X}^{\chi}(\pi) \ge
        \begin{cases}
      \dim X+2,  \quad &\text{if $\bG(X)=\rU(X)$ and $\epsilon_{\chi}\equiv \dim X \pmod 2$};\\
      \dim X + 1,  \quad &\text{if $\bG(X)=\rO(X)$ with $\dim X$ odd};\\
      \dim X +3,  \quad &\text{if $\bG(X)=\Mp(X)$}.
      \end{cases}
    \end{equation*}
  \end{enumerate}
\end{thm}
% \begin{proof}
%   Parts~\eqref{item:L1-pole},~\eqref{item:Lhalf-non-van} are clear by Theorems~\ref{thm:LO-equiv-eis-max-pole},~\ref{thm:generic-L-pole-Eis-pole}. They also show the relations ``$\ge$'' in parts~\eqref{item:L1-no-pole},~\eqref{item:Lhalf-van}.

% By the computation in Section~4 of \cite{Furusawa-MR1353315}, we see that the theta lift of a $\psi$-generic $\pi$ to the  $\bG(Y)$ with $Y$ a maximally split $\epsilon$-Hermitian space and $\dim Y$ given by the RHS of the equalities in parts~\eqref{item:L1-no-pole},~\eqref{item:Lhalf-van} is $\psi^{-1}$-generic. In particular, it is non-vanishing. % Note that our definition of theta lift has a complex conjugate which \cite{Furusawa-MR1353315} does not have.
% Thus we get equality.
%[compute Whittaker coefficients. This is quite involved. Do this next time.]
% \end{proof}
% \begin{rmk}
%   Our results do not obtain precise information on the lowest occurrence when $L(s,\pi\times\chi^{\vee})$ vanishes at $s=\half$.
% \end{rmk}
\begin{rmk}
  By the conservation relation for local theta correspondence \cite{Sun-Zhu-MR3369906}, there always exists an $\epsilon$-Hermitian space $Z_{[v]}$ over $E_{v}$ of dimension given by the RHS of the equalities in items~\eqref{item:L1-pole},~\eqref{item:Lhalf-non-van}  such that the local theta lift of $\pi_{v}$  to $\bG(Z_{[v]})$ is non-vanishing.
  Thus in the case of items~\eqref{item:L1-no-pole},~\eqref{item:Lhalf-van} and $\bG(X)\neq \rO(X)$,  the collection $\{Z_{[v]}\}_{v}$ for $v$ running over all places of $F$ is always incoherent, i.e., there does not exist  an $\epsilon$-Hermitian space $Z$ over $E$ such that the localisation $Z_{v}$ is isomorphic to $Z_{[v]}$ for all $v$.
  In the case of items~\eqref{item:L1-no-pole},~\eqref{item:Lhalf-van} and $\bG(X)= \rO(X)$, we have a non-trivial theta lift of $\pi_{v}\otimes(\chi_{v}\circ\upsilon_{v})\otimes (\eta_{[v]}\circ\det)$ to $\bG(Z_{[v]})$ for $\eta_{[v]}$ being the trivial character or the sign character for each place $v$ of $F$, but the collection $\{\eta_{[v]}\}_{v}$ is incoherent, i.e., there does not exist an automorphic character $\eta$ of $\AA_{F}^{\times}$ such that the localisation $\eta_{v}$ is equal to $\eta_{[v]}$ for all $v$.
See the definitions of first occurrence \eqref{eq:FO-defn} and lowest occurrence \eqref{eq:LO-defn-O} for $\rO(X)$ for why we have a $(\chi_{v}\circ\upsilon_{v})$-twist.
  We also note that when $\pi$ is an \ICARep{} and $L(\half,\pi\times\chi^{\vee})=0$, it is conjectured that there is an arithmetic version of the Rallis inner product formula which says that the conjectural Beilinson--Bloch height pairing of arithmetic theta lifts (which are cycles on Shimura varieties constructed from an incoherent collection of $\epsilon$-Hermitian spaces) gives the derivative $L'(\half,\pi\times\chi^{\vee})$ up to some ramified factors and some abelian $L$-functions.
  The low rank cases have been proved in \cite{Kudla-Rapoport-Yang-book-MR2220359} and \cite{Yifeng-Liu-MR2928563,Yifeng-Liu-MR2928564}.
  More recently, the  cases of unitary groups of higher rank have been proved in \cite{Li-Liu-L-derivative-I-MR4334978} and \cite{Li-Liu-L-derivative-II-MR4390300}, conditional on hypothesis of the modularity of Kudla’s generating
functions of special cycles.
\end{rmk}

In terms of  `$(\chi,b)$'-factors, we have
\begin{thm}\label{thm:chi-b-generic}
  Let  $\bG(X)=\rU(X)$ with $\epsilon_{\chi}\not\equiv \dim X \pmod 2$, $\bG(X)=\rO(X)$ with $\dim X$ even or $\bG(X)=\Sp(X)$. Let $\pi$ be a cuspidal member in a generic global $A$-packet of $\bG(X)(\AA)$. Let $\chi$ be a conjugate self-dual automorphic character of $\R_{E/F}\GL_{1}(\AA)$. Then 
  the following are equivalent.
  \begin{enumerate}
  \item The global $A$-parameter $\phi$ of $\pi$ has a $(\chi,1)$-factor.
  \item The complete $L$-function $L(s,\pi\times\chi^{\vee})$ has a pole at $s=1$ (and this is its maximal pole).
  \item The Eisenstein series $E^{Q_{1}}(g,f_{s})$ has a pole at $s=1$ for some choice of $f_{s}\in \cA^{Q_{1}}(s,\chi\boxtimes\pi)$ (and this is its maximal pole).
  \item The lowest occurrence index $\LO_{X}^{\chi}(\pi)$ is
    \begin{equation*}
      \begin{cases}
        \dim X - 1, \quad &\text{if $\bG(X)=\rU(X)$};\\
        \dim X - 2, \quad &\text{if $\bG(X)=\rO(X)$};\\
        \dim X, \quad &\text{if $\bG(X)=\Sp(X)$}.
      \end{cases}
    \end{equation*}
  \end{enumerate}
\end{thm}

\begin{rmk}
  The statements that the poles are maximal are automatic since $\pi$ lies in a  generic global $A$-packet.
  We note that when $\bG(X)=\rU(X)$ with $\epsilon_{\chi}\equiv \dim X \pmod 2$, $\bG(X)=\rO(X)$ with $\dim X$ odd or $\bG(X)=\Mp(X)$, $\phi_{\pi}$ cannot have a $(\chi,1)$-factor as the parity condition  is not satisfied. 
\end{rmk}

%bib
\bibliography{../tex/MyBib-20220401}
\bibliographystyle{alpha}

 \end{document}